\documentclass[11pt,reqno,a4paper]{article}
\usepackage{amsmath}
\usepackage{amsthm}
\usepackage{amssymb}
\usepackage{graphicx}
\usepackage{verbatim}

\newtheorem{theorem}{Theorem}[section]
\newtheorem{lemma}[theorem]{\sc{Lemma}}
\newtheorem{proposition}[theorem]{Proposition}
\newtheorem{corollary}[theorem]{Corollary}

\theoremstyle{definition}
\newtheorem{definition}{Definition}[section]

\theoremstyle{remark}
\newtheorem{remark}{Remark}[section]

\renewcommand{\phi}{\varphi}

\renewcommand{\leq}{\leqslant}
\renewcommand{\geq}{\geqslant}

\newcommand{\sca}[2]{\langle #1 | #2\rangle}

\newcommand{\opnorm}[1]{{\left\| #1 \right\|}_{\mathcal{L}}}

\renewcommand{\d}{\mathrm{d}}

\newcommand{\Medial}{\mathscr{M}}

\DeclareSymbolFont{bbold}{U}{bbold}{m}{n}
\DeclareSymbolFontAlphabet{\mathbbold}{bbold}

\newcommand{\R}{\mathbb{R}}

\newcommand{\E}{\mathbb{E}}
\newcommand{\T}{\mathbb{T}}

\newcommand{\C}{\mathbb{C}}

\usepackage{fix-cm}
\usepackage[utf8]{inputenc} 
\usepackage[english]{babel} 
\usepackage[T1]{fontenc}
\usepackage{graphicx}
\usepackage{amsmath}
\usepackage{amssymb}
\usepackage{mathrsfs}
\usepackage{mathtools}
\usepackage{booktabs} 
\usepackage{caption} 
\usepackage{subcaption} 
\usepackage{pgfplots}
\usepackage[all]{nowidow}
\usepackage{tikz}
\usetikzlibrary{er,positioning,bayesnet}
\usepackage{multicol}
\usepackage{algpseudocode,algorithm,algorithmicx}
\usepackage{xfrac}
\usepackage{fancyhdr}
\usepackage[left=3cm, right=3cm, top=2cm, bottom=3cm]{geometry} 
\usepackage{float}
\usepackage{hyperref}
\hypersetup{
    colorlinks=true,
    linkcolor=blue,
    urlcolor=blue,
    citecolor=blue,}
\usepackage{enumitem} 
\usepackage{bm} 
\usepackage{fancybox} 
\usepackage{color}
\usepackage{xcolor} 

\renewcommand{\C}{\mathcal{C}}
\renewcommand{\H}{\mathbb{H}}

\newcommand{\reach}{\mathrm{reach}} 
\renewcommand{\d}{\mathrm{d}}

\title{Submanifolds of class $C^{1,\alpha}$ and sets with positive $\mu$-reach}

\author{
Vincent Borrelli\textsuperscript{1}, Jean-Baptiste Follet\textsuperscript{1,2} and Boris Thibert\textsuperscript{2}
}

\begin{document}
\maketitle
\begin{center}
\emph{\textsuperscript{1} Université Claude Bernard Lyon 1, CNRS, Institut Camille Jordan,\\ F-69622 Villeurbanne, France}.\\
\emph{\textsuperscript{2} Université Grenoble Alpes, CNRS, Grenoble INP, Laboratoire Jean Kuntzmann,\\ 38000 Grenoble, France}.\\
\end{center}
\begin{abstract}

It is well-known since the seminal work of Herbert Federer [Trans. of the AMS, 1959] that submanifolds of class $C^{1,1}$ have positive reach. In this paper, we extend this property to less regular submanifolds by using the notion of $\mu$-reach that was introduced in the 2000's. We first show that every compact $C^1$ submanifold of the Euclidean space $\E^n$ has positive 
$\mu$-reach for all $\mu<1$. We then show that intermediate regularities $C^{1,\alpha}$ induce more quantitative results on the norm $\|\nabla \d_M\|$ of the generalized gradient of the distance function~$\d_M$ to the submanifold. More precisely, if $M\subset \E^n$ is a submanifold of class $C^{1,\alpha}$, with $\alpha<1$, then there exists a constant $C>0$ such that 
 $$\forall p\in\E^n\setminus M,\quad 1 - \| \nabla \d_M(p) \|^2  \leq C ~ \d_M(p)^{\frac{2 \alpha}{1- \alpha}}.$$
We finally show that the exponent $2\alpha/(1-\alpha)$ in this estimate is sharp.
\end{abstract}

\tableofcontents

\section{Introduction}
Riemannian manifolds of dimension $n$ cannot, in general, be 
$C^2$-isometrically embedded into the Euclidean space $\E^{n+1}$ of dimension $n+1$. For instance, the Hilbert–Efimov theorem asserts that the hyperbolic plane $\H^2$ does not admit any $C^2$ isometric embedding into $\E^3$~\cite{efimov}. Similarly, curvature-based arguments show that the two-dimensional flat torus 
$\T^2$ does not admit any 
$C^2$ isometric embedding into 
$\E^3$. In contrast, the 
$C^1$ isometric embedding theorem of Nash and Kuiper establishes the existence of infinitely many isometric embeddings of a Riemannian manifold when the regularity is relaxed to the $C^1$ setting~\cite{n-c1ii-54,k-oc1ii-55,kuiperII}.
Despite these foundational results, the geometric analysis of such 
$C^1$ submanifolds remains incomplete, and a deeper understanding of their extrinsic geometry  
is still needed. In this article, we investigate the geometry of $C^1$ and $C^{1,\alpha}$ submanifolds of 
$\E^n$ using the notion of sets with positive 
$\mu$-reach, introduced by Chazal, Cohen-Steiner, and Lieutier in 2006~\cite{chazal2006sampling,Chazal-Lieutier}, which generalizes Federer's classical concept of sets with positive reach.

\paragraph{Sets with positive reach and $C^{1,1}$ regularity.} Sets with positive reach were introduced by Herbert Federer in 1959 in order to generalize the notion of curvature measures to non smooth and non convex objects~\cite{Federer59}. On the one hand, this concept makes it possible to define notions of curvature on non regular objects. On the other hand, it provides a tool - the projection map - for comparing sets that are Hausdorff close, thereby enabling stability results. As a consequence, it plays a central role in several areas of geometry, including Geometric Measure Theory or Geometric Inference~\cite{fu1993convergence,cohen2006second,chazal2009stability}. For further details, we refer to the survey of Th\"{a}le~\cite{thale200850}. 

By definition, a closed set $K$ of $\E^n$ has a reach  strictly greater than $r>0$ if every point~$p$ that is at a distance less than $r$ from $K$, has a unique closest point on $K$. This notion is intimately related to the distance function $\d_K:\E^n\to\R$ to the set $K$ and allows to define the $r$-offset of $K$:
\[K^r=\{p\,|\, d_K(p)\leq r\}\]
also called the tubular neigborhood of $K$ of size $r.$
Federer studied in detail the  function~$\d_K$ and showed that for $r>0$ small enough, the level set $\d_K^{-1}(r)=\partial K^r$ is a codimension one submanifold of class $C^{1,1}$, \textit{i.e.} of class $C^1$ with a Lipschitz normal.

Federer also considered the case of sets $K$ that are submanifolds of $\E^n$. He mentioned without giving any detail that the graph of a Lipschitz function has positive reach if and only if it is of class $C^{1,1}$~\cite[Remark 4.20]{Federer59}. This statement happens to be not straightforward and has generated several works in the last decades. Lytchak showed in 2005 that a topologically embedded manifold without boundary with positive reach  is of class $C^{1,1}$~\cite{lytchak2005almost}. Scholtes gave an alternative proof in 2013  in the  case of codimension one submanifolds homeomorphic to the sphere~\cite{scholtes2013hypersurfaces}. Rataj and Zajíček also provided a proof in 2017, under the restrictive assumption that the manifold is Lipschitz~\cite{rataj2017structure}.
 More recently in 2024, using an approach based on homology, Lieutier and Wintraecken showed that a closed submanifold has positive reach if and only if it is locally the graph of a $C^{1,1}$ function~\cite{lieutier2024manifolds}. 
 
\paragraph{Sets with positive $\mu$-reach.} Sets with positive reach present two main limitations. First, although they were introduced to extend curvature measures to irregular objects, they do not include simple irregular structures such as triangulations. Second, they are not stable with respect to the Hausdorff distance. Consequently, they are not directly suited to the framework of \textit{Geometric Inference}, where one aims to infer geometric properties of a shape $K$ from a Hausdorff approximation. To overcome these limitations, Chazal, Cohen-Steiner, and Lieutier introduced in 2006 the notion of sets with positive 
$\mu$-reach~\cite{chazal2006sampling,Chazal-Lieutier}.

The notion of $\mu$-reach can be viewed as a filtered notion of reach, derived from the gradient of the distance function $\d_K$.  This distance function is 1-Lipschitz and is therefore differentiable almost everywhere. However it is not differentiable on the whole space $\E^n$. It can be shown that the reach of a compact set $K\subset\E^n$ can be defined as the largest radius $r$ for which the distance function $\d_K$ is differentiable on the $r$-offset $K^r$~:
\[\reach(K) := \sup \{r ~ \vert ~  \d_K\mbox{ is differentiable on }K^r\setminus K \},\]
see Section~\ref{sec:distancefunction}.
In 2004, Lieutier introduced an extension of the gradient map $\nabla \d_K$ to the entire Euclidean space, defining a \textit{generalized gradient}~\cite{LIEUTIER20041029}. 
This generalized gradient satisfies $\|\nabla\d_K(p)\|\leq 1$ with equality if and only if $\d_K$ is differentiable at $p$. This observation naturally leads to a refined version of the reach, namely the $\mu$-reach of $K$, for  $0\leq \mu\leq 1$, which is defined by:
\[
\reach_\mu(K) := \sup \{r ~ \vert ~ \forall p\in K^r, ~ \| \nabla \d_K(p)\| \geq \mu \text{ on } K^r\setminus K \}.
\]
The class of sets with positive $\mu$-reach contains all the sets with positive reach and also more irregular objects such as triangulations. Furthermore, it offers a natural tool in the field of Geometric Inference: stability results of the normals and of the curvature measures have been obtained for sets with positive $\mu$-reach~\cite{chazal2009normal,chazal2009stability}. 

\paragraph{Contributions.} As mentioned above, Federer made a connection between submanifolds of class $C^{1,1}$ and submanifolds with positive reach. In this article, we extend one implication of Federer's claim and investigate the geometry of $C^1$ and $C^{1,\alpha}$ submanifolds of $\E^n$ using the notion of sets with positive $\mu$-reach. We first show the following result. 
\begin{theorem} \label{theo: C1_implique_continuite_fonction_crit}
Let $M$ be a $C^1$ compact submanifold of $\E^n$ without boundary. Then $M$ has positive $\mu$-reach for any $\mu < 1$. \end{theorem} 
This result is optimal for the parameter $\mu$. Indeed, any manifold $M$ with  positive $\mu$-reach for $\mu=1$, has positive reach, hence is of class $C^{1,1}$. Theorem~\ref{theo: C1_implique_continuite_fonction_crit} is obviously true if the compact manifold is of dimension $0$, since it has positive reach in that case. Henceforth, the manifold is assumed to have dimension greater than $1$. Using results of~\cite{fu1985tubular,chazal2007shape}, Theorem~\ref{theo: C1_implique_continuite_fonction_crit} directly implies a regularity result on the boundaries of the offset $M^r:=\d_M^{-1}([0,r])$ and of the double offset $M^{r,t}:=\d_L^{-1}([0,t])$, where $L=\overline{\E^n\setminus M^r}$, as stated in the following corollary.
\begin{corollary}\label{coro:offsetshavepositivereach}
Let $M$ be a $C^1$ compact submanifold of $\E^n$ without boundary.  Then for any $\mu<1$ and any $0<r<\reach_\mu(M)$, we have: 
\begin{itemize}
\item The set $L$ has a reach greater than $\mu r$ and its boundary $\partial L=\partial M^r$ is a Lipschitz manifold.
\item For any $t\in]0,\mu r[$ the boundary of the double offset $\partial M^{r,t}$ is of class $C^{1,1}$ and has a reach that satisfies  $\reach(\partial M^{r,t}) \geq \min(t,\mu r - t)$.
\end{itemize}
\end{corollary}

Since Federer's curvature measures are well defined on sets with positive reach, this corollary shows that they can also be defined on the sets $\d_M^{-1}([r,+\infty[)$. Furthermore, stability results such as those established in~\cite{chazal2009stability}, open the possibility of estimating Federer's curvature measures on the level sets $\d_M^{-1}(r)$ from a Hausdorff approximation of the manifold $M$. This approach may be applied, for instance, to the study of the geometry of the $C^1$ isometric embeddings mentioned above, where each $C^1$
isometrically embedded submanifold $M$ arises as the limit of a sequence of smooth submanifolds $(M_n)_{n\geq 0}$.\\

\noindent 
Theorem~\ref{theo: C1_implique_continuite_fonction_crit} provides a non quantitative result showing that the norm $\|\nabla \d_M(p)\|$ of the generalized gradient tends to $1$ when $p$ approaches the submanifold~$M$. When considering an intermediate regularity $C^{1,\alpha}$ between $C^1$ and $C^{1,1}$, one obtains a more quantitative statement, namely a lower bound of $\|\nabla \d_M\|$ on the level sets of $\d_M$. In particular, the rate at which $\|\nabla \d_M(p)\|$ converges to 1 depends on the $C^{1,\alpha}$ regularity;  this convergence becomes faster as $\alpha$ increases, as described in the following theorem.

\begin{theorem} \label{theo: majoration_fonction_critique}
Let $M$ be a $C^{1,\alpha}$ compact submanifold of $\E^n$ without boundary with $0\leq \alpha \leq 1$.
Then there exists a constant $C > 0$ such that, for all $p \in \E^n \setminus M$,
$$  1 - \| \nabla \d_M(p) \|^2  \leq C ~\d_M(p)^{2 \alpha} \Big( 1 - \| \nabla \d_M(p) \|^2 \Big)^{\alpha}.$$
In particular, 
\begin{itemize} 
\item if $\alpha <1$, then for every $p \in \E^n \setminus M$ 
$$
 1 - \| \nabla \d_M(p) \|^2  \leq C^{\frac{1}{1-\alpha}} ~ \d_M(p)^{\frac{2 \alpha}{1- \alpha}},
$$
\item if $\alpha=1$, then for every $p\in \E^n$ such that $\d_M(p)< 1/\sqrt{C}$, we have 
$$\|\nabla\d_M(p)\|=1.$$
\end{itemize}
\end{theorem}
 We show in Section~\ref{sec: Deux_exemples} that, when $\alpha\neq 1$, the exponent $2\alpha/(1-\alpha)$ in Theorem~\ref{theo: majoration_fonction_critique} is sharp. When $\alpha=1$, we also recover the property that a $C^{1,1}$-submanifold $M$ has positive reach, hence $\|\nabla\d_M(\cdot)\|=1$ in a neighborhood of $M$.\\

The paper is organized as follows: In Section~\ref{sec:distancefunction}, we recall classical results on the distance function, as well as the definitions of reach and $\mu$-reach. Section~\ref{sec:proofTheoremC1}, is devoted to the proof Theorem~\ref{theo: C1_implique_continuite_fonction_crit}. In Section~\ref{sec:theoremC1alpha}, we recall the definition of submanifolds of class $C^{1,\alpha}$ and establish several technical results needed for the proof of Theorem~\ref{theo: majoration_fonction_critique} which presented in Subsection~\ref{subsec:profftheme2}. Finally, in Section~\ref{sec: Deux_exemples}, we show the sharpness of the exponent in Theorem~\ref{theo: majoration_fonction_critique}.

\section{Distance function to a closed set}\label{sec:distancefunction}
The distance function $\d_K : \E^n \to \R$ to a closed subset $K\subset \E^n$ is defined for every point $p \in \E^n$ by the following formula
\[ \d_K(p) := \underset{q \in K}{\inf} \| p - q \|. \]
In this section, we review regularity properties of the distance function $\d_K$ and introduce the notions of medial axis, reach and $\mu$-reach. One may refer to the seminal paper of Herbert Federer in 1959 that introduces the notion of reach and studied the distance function in order to generalize the notion of curvature measures to non smooth and non convex objects~\cite{Federer59}. For the notions of generalized gradient and notion of $\mu$-reach, one may refer to~\cite{LIEUTIER20041029,chazal2006sampling}.
\subsection{Projections, medial axis and reach}
Let $p \in \E^n$. We call projection of $p$ on $K$ any point that realises the minimum of $\d_K$ and we denote by $\Gamma_K(p)$ the set of all projections of $p$ on $K$, namely
$$\Gamma_K(p) = \{ q \in K ~|~ \d_K(p) =  \| p - q \| \}. $$
Since $K$ is a closed subset of $\E^n$, we have $\Gamma_K(p) \neq \emptyset$. Furthermore, since $\Gamma_K(p)$ is included in a sphere of center $p$ and 
of radius $\d_K(p)$, it is bounded and thus compact. 
In particular, there exist two points $q_1, q_2 \in \Gamma_K(p)$ that realize the diameter of $\Gamma_K(p)$.

\begin{definition}
The medial axis $\mathscr{M}(K)$ of $K$ is the set of points of $\E^n$ that have at least two projections on $K$, i.e.
    \[ \mathscr{M}(K) := \big\{ p \in \E^n \setminus K ~ | ~ \mathrm{Card} \big( \Gamma_K(p) \big) \geq 2 \big\}.  \]
\end{definition}

\noindent 
Any point on the complement of the medial axis has a single projection, which allows to define the projection map 
$$\pi_K:\E^n \setminus \mathscr{M}(K) \to K$$ 
by $\pi_K(p) = q$, where $\Gamma_K(p)=\{q\}$.
We denote by $K^r = \{ x \in \E^n ~\vert ~ \d_K(x) \leq r \}$ the $r$-offset of $K$ and we define the reach the following way :

\begin{definition} \label{def: reach} The reach of $K$ is the radius of the larget offset on which $\pi_K$ is defined, \textit{i.e.}
$$
\reach(K) = \text{sup} \{ r\geq 0 ~ \vert ~ \pi_K ~\text{is defined on}~ K^r \}.
$$
\end{definition}

\subsection{Generalized gradient of the distance function}
It can be easily shown that the distance function $\d_K$ to a closed set $K\subset \E^n$ is $1$-Lipschitz and is thus by Rademacher theorem differentiable almost everywhere. The following proposition indicates that it is differentiable on the complement of the medial axis.

\begin{figure}[H]
       \centering
       \begin{subfigure}[b]{0.49\textwidth}
        \centering  \includegraphics[width=1\linewidth]{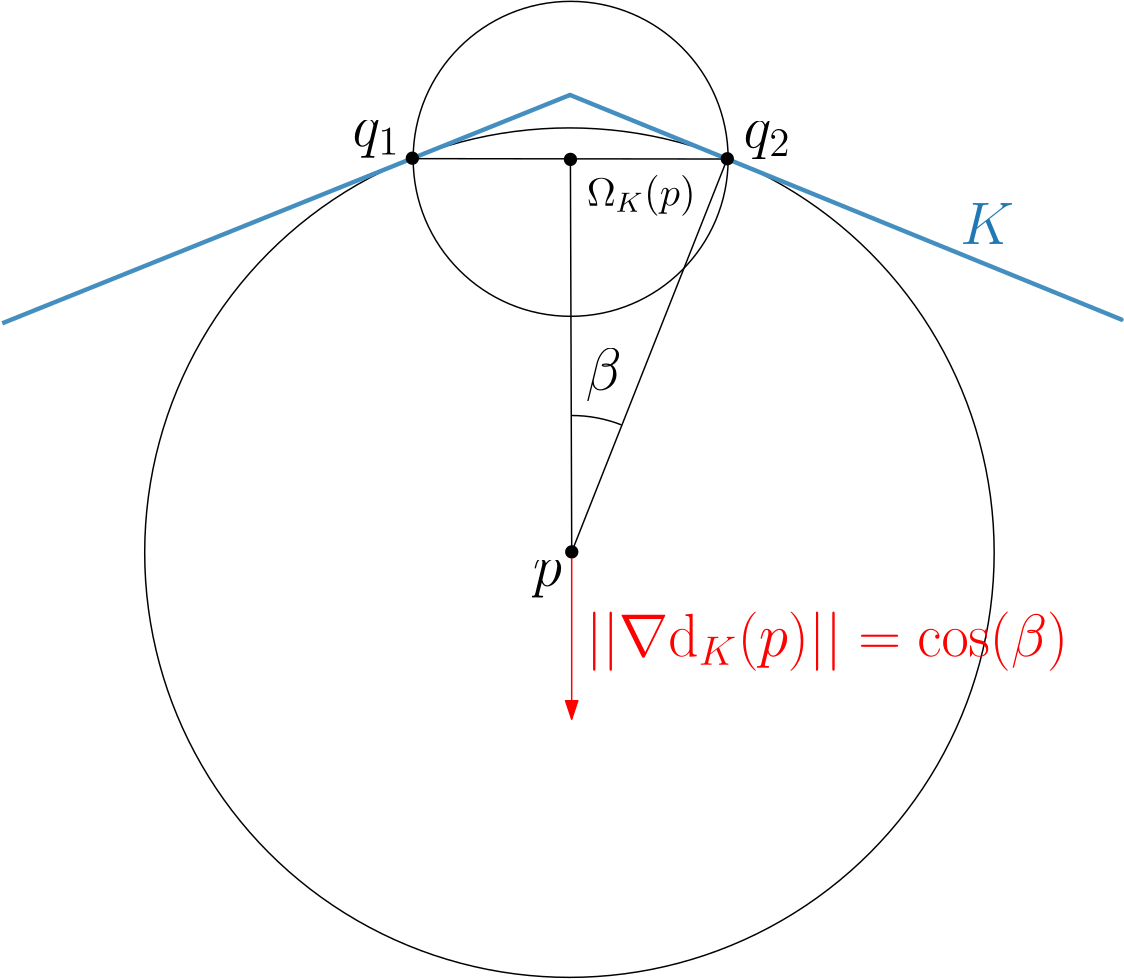}
        \caption{\centering Generalized gradient $\nabla \d_K$ at a point $p$ of the medial axis $\Medial(K)$ with $\Gamma_K(p) =  \{q_1,q_2\}$}
       \end{subfigure}
       \hfill
       \begin{subfigure}[b]{0.49\textwidth}
         \centering \includegraphics[width=1\linewidth]{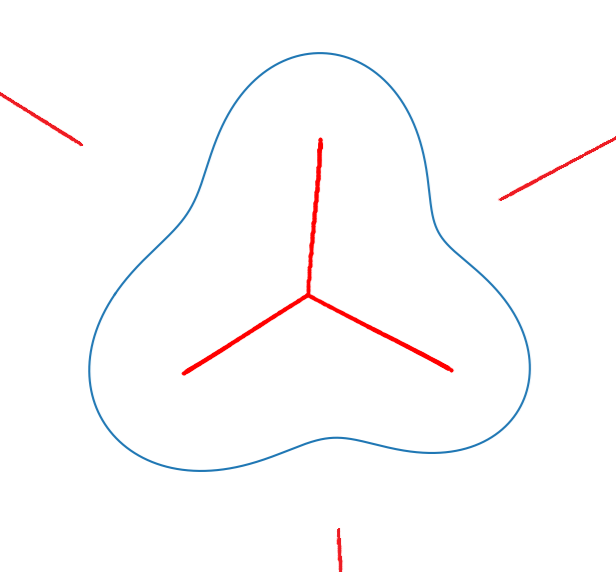}
         \caption{\centering A closed set $K$ (blue) and its medial axis $\Medial(K)$ (red).}
       \end{subfigure}
      \caption[]
        {\centering Generalized gradient of distance function $\d_K$.}
      \label{fig: gradient_generalise}
\end{figure}

\begin{proposition} \label{prop: distance_differentiable}
Let $K$ is a closed subset of $\E^n$. Then for any $p \in \E^n \setminus K$, the following assertions are equivalent  
\begin{enumerate}
    \item $\d_K$ is differentiable at $p$,
    \item The projection $\pi_K$ is defined at $p$.
\end{enumerate}
Furthermore for any $p\in \E^n \setminus \Medial(K)$ one has
$$\nabla \d_K (p) = \frac{p - \pi_K(p)}{\|p - \pi_K(p)\|}.$$
\end{proposition}
A notion of generalized gradient of the distance function $\d_K$ that is also defined on the medial axis $\Medial(K)$ was introduced by André Lieutier in 2004~\cite{LIEUTIER20041029}.
\begin{definition} \label{def: gradient_generalise}
Let $K$ be a closed subset of $\E^n$. 
The generalized gradient of $\d_K$ at  $p \in \E^n \setminus K$ is defined by
$$\nabla \d_K (p) := \frac{p - \Omega_K(p)}{\d_K(p)}$$
where $\Omega_K(p)$ is the center of the smallest ball containing $\Gamma_K(p)$ (see Figure \ref{fig: gradient_generalise}).
\end{definition}
Remark that the generalized gradient $\nabla \d_K$ coincides with the classical one at points where $\d_K$ is differentiable.  In the sequel, we denote by $\mathbb{B}_K(p)$ the smallest ball containing $\Gamma_K(p)$ and by $\mathscr{R}_K(p)$ its radius. Its boundary sphere is denoted $\mathbb{S}_K(p)$.
The following relation will be useful further on (see \cite{LIEUTIER20041029}).
\begin{proposition} \label{prop: expression_grad_g_rayon}
Let $K$ be a closed subset of $\E^n$ and let $p \in \E^n \setminus K$ then
$$ \| \nabla \d_K (p) \| = \sqrt{ 1 - \bigg( \frac{\mathscr{R}_K(p)}{\d_K(p)} \bigg)^2} .$$
\end{proposition}

\begin{remark}\label{rem:reachrecast} We see from the definition of the generalized gradient that $\d_K$ is differentiable at a point $p \in \E^n \setminus K$ if and only if $\| \nabla \d_K(p) \| = 1$. This allows to recast the medial axis as the set of points where $\d_K$ is not differentiable, and to recast the reach of $K$ as
$$
\reach(K) = \text{sup} \{ r\geq 0 ~ \vert ~ \forall p\in K^r\setminus K  ~\|\nabla \d_K(p)\| =1\}.
$$
\end{remark}

\subsection{Critical function and $\mu$-reach}
The notions of critical function and $\mu$-reach were introduced in 2006 in the context of geometric inference in order to get stability results~\cite{chazal2006sampling,Chazal-Lieutier}. As mentioned in Remark~\ref{rem:reachrecast}, the reach corresponds  the size of the offset where the generalized gradient is of norm $1$. This suggests the definition of a relaxed version of the reach that is called $\mu$-reach. 
\begin{definition}
Let $\mu \leq 1$. The $\mu$-reach of a closed set $K \subset \E^n$ is defined by:
$$\reach_\mu(K) := \sup \{r ~ \vert ~ \forall p\in K^r ~ \| \nabla \d_K(p)\| \geq \mu \text{ on } K^r \}.$$
\end{definition}
In order to study the behavior of the norm of the generalized gradient when we get closer to the set $K$, it will be convenient to use the critical function defined below. This function scans the level sets of the distance function by looking for the smallest value of the generalized gradient.
\begin{definition}
The critical function $\chi_K : ]0,+\infty[ \to \R_+$ of a closed set $K \subset \E^n$ is the real function defined by
$$\chi_K(d) := \underset{p \in \partial K^d}{ \text{inf} }\vert \vert \nabla \d_K(p) \vert \vert$$
with $\partial K^d = \{ x \in \E^n ~|~ \d_K(x) = d \}$ the  boundary of the $d$-offset of $K$.
\end{definition} 
\noindent Obviously, the $1$-reach corresponds to the reach, namely we have $\reach_1 (K) = \reach(K)$.

\begin{figure}[H]
       \centering
       \begin{subfigure}[b]{0.42\textwidth}
        \centering  \includegraphics[width=1\linewidth]{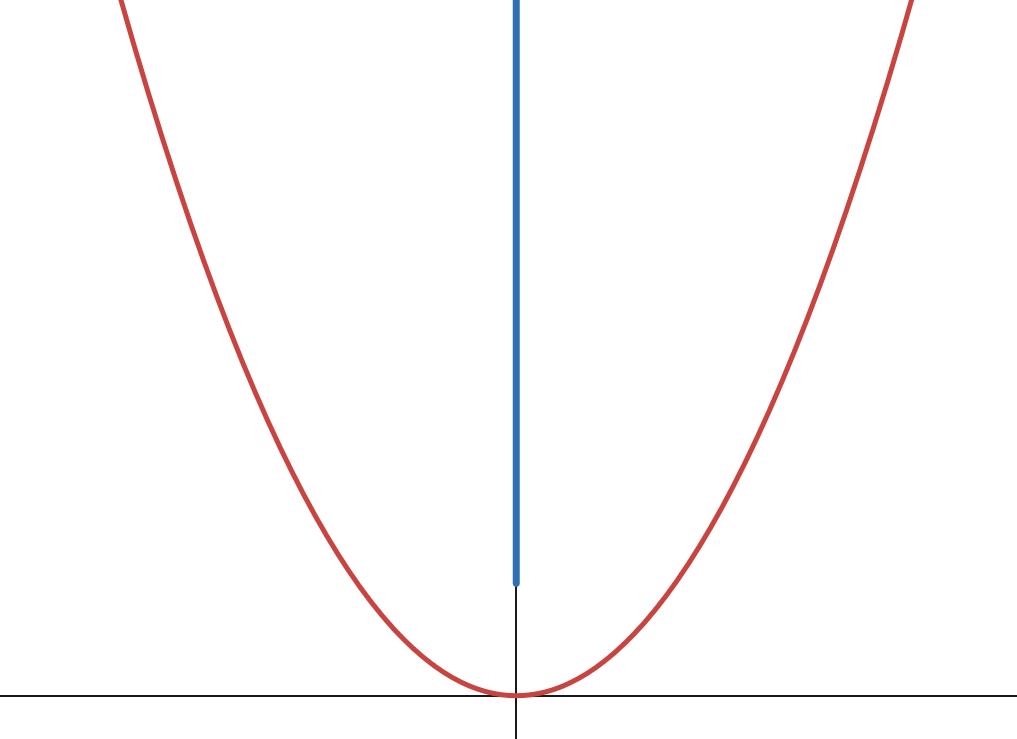}
       \end{subfigure}
       \hfill
       \centering
       \begin{subfigure}[b]{0.57\textwidth}
        \centering  \includegraphics[width=1\linewidth]{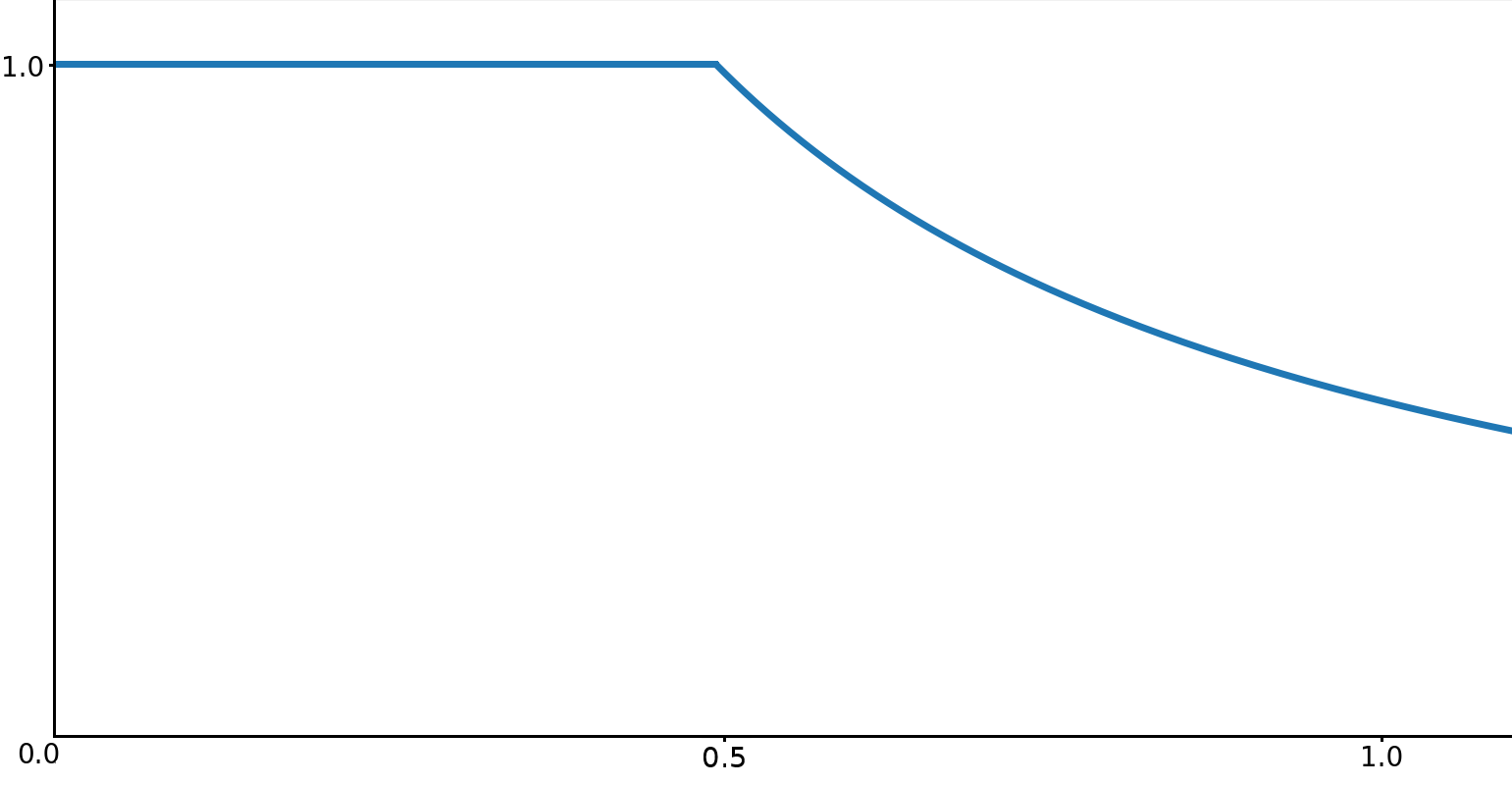}
       \end{subfigure}
      \caption[]
        {\centering Left: the parabola $x^2$ and its medial axis (blue). Right: graph of its critical function. Since $\reach(K) = \frac{1}{2}$ we have $\chi(d)=1$ for $d\leq \frac{1}{2}$.}
      \label{fig: exemple_critical_function}
\end{figure}

\section{On the $\mu$-reach of $C^1$ submanifolds}\label{sec:proofTheoremC1}
Recall that Theorem~\ref{theo: C1_implique_continuite_fonction_crit} states that if $M$ is a compact $C^1$ submanifold of $\E^n$ without boundary, then it has positive $\mu$-reach for any $\mu < 1$. 
The goal of this section is to prove this theorem and its corollary. Before that, we make two comments. 

\paragraph{Compactness assumption of Theorem~\ref{theo: C1_implique_continuite_fonction_crit}.}
The compactness assumption in Theorem~\ref{theo: C1_implique_continuite_fonction_crit} is necessary.
Indeed, for any $\mu\, \in ]0,1[$, we can construct a  $C^1$ curve $M$ in $\E^2$ such that for all $\mu' > \mu$, $\reach_{\mu'}(M) = 0$.
The idea of the construction is to take a triangle wave of amplitude 1 and period $T = 2  \sqrt{(1/\mu^2)-1}$ and  to replace neighborhoods of its vertices by arc  of  circles whose radius decrease the further you get from the origin and tend to zero at infinity, see Figure~\ref{fig:contre_exemple_non_compact}.

\begin{figure}[!ht]
    \centering
    \includegraphics[width=0.95\linewidth]{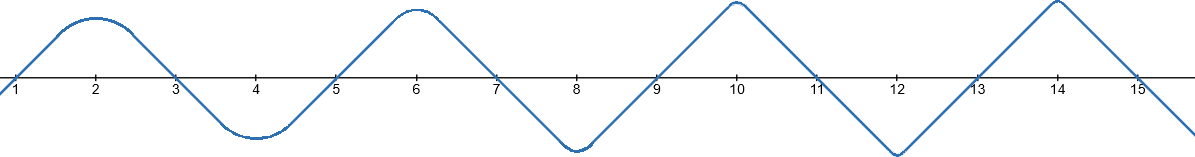}
    \caption{Example of a non compact $C^1$ curve whose $\mu$-reach vanishes for $\mu > \frac{1}{2}$.}
    \label{fig:contre_exemple_non_compact}
\end{figure}

To be more precise, for each triangle, the medial axis is a half line starting at the center of the quarter of circle. On the part of the medial axis inside each triangle, the norm of the generalized gradient is equal to $\mu$. However, as the radius of the quarter circle edge of that triangle get smaller, the medial axis get closer to the curve. Thus the medial axis touches the curve at the limit. Therefore, for all $\mu' > \mu$, $\reach_{\mu'}(M) = 0$.
\paragraph{Converse of Theorem~\ref{theo: C1_implique_continuite_fonction_crit}.} 
It is natural to wonder whether the converse of Theorem~\ref{theo: C1_implique_continuite_fonction_crit} holds, namely to ask the following question: if $M$ is a compact manifold whose $\mu$-reach is positive for any $\mu<1$, does it imply that it is of class $C^1$ ? 
Recently, for any any $\mu \in ]0,1[$, Antoine Commaret built an example of a set $K_\mu$ that yet has a positive $\mu$-reach and which is non-rectifiable, hence not of class $C^1$~\cite[pp 31 to 37]{TheseCommaret}. The question of building a non $C^1$ compact manifold $K$ that has positive $\mu$-reach for any $\mu<1$ is still open.


\subsection{Proof of Theorem~\ref{theo: C1_implique_continuite_fonction_crit}}
Before going to the proof of Theorem~\ref{theo: C1_implique_continuite_fonction_crit}, we need two lemmas.

\begin{lemma} \label{lem_normale_distnace_min_variete}
Let $M$ be a $C^1$ submanifold of $\E^n$  and $p \in \E^n \setminus M$.
For any $q \in\Gamma_M(p) \setminus \partial M$, the line joining $p$ and $q$ is normal to $M$ at $q$.
\end{lemma}

\begin{proof}
Let $t\in T_q M$ be any tangent vector to $M$ at $q$. Since $M$ is of class $C^1$ and $q$ belongs to the interior of $M$, there exists $\varepsilon>0$ and a curve $\gamma:]-\varepsilon,\varepsilon[\to M$ of class $C^1$ such that $\gamma(0) = q$ and $\gamma'(0) = t$. Since $q$ is a projection of $p$, the map $f(x):=\|\gamma(x)-p\|^2$ has a minimum at $x=0$, which implies that 
$$
0=f'(0) = 2 \sca{\gamma(0)- p}{\gamma'(0)}=\sca{q-p}{t},
$$
hence the result.
\end{proof}

\begin{remark}
The hypothesis $q \notin \partial M$ is necessary. 
As a counterexample, consider $M := [0,1] \times \{ 0 \}$ in $\E^2$ and $p = (0,2)$. Then $q= (0,1)$, and thus the line connecting $p$ and $q$ is not normal to $M$.
\end{remark}


We recall the radius of a set $X \subset \E^n$ is the radius of the smallest ball enclosing X, while the diameter of $X$ is the maximal distance between two points in $X$.
Jung's theorem \cite{Jung1901}, is the following inequality between the radius and the diameter:

\begin{theorem}[Jung 1901] \label{theo: Jung}
    If $X$ is a subset of $\E^n$, then $$ \sqrt{2\big( 1 + \frac{1}{n}\big)} ~ \mathrm{Radius}(X) \leq \mathrm{Diam}(X)$$
\end{theorem}

\begin{lemma} \label{lem: consequnce_Jung}
    Let $K$ be a closed subset of $\E^n$, $p \in \mathscr{M}(K)$ and $q_1$ and $q_2$ two projections of $p$ on $K$ that realize the diameter of $\Gamma_K(p)$.
    Then, the generalized gradient satisfies
     $$\| \nabla \d_K(p) \|^2 \geq 1 - \frac{ 1 - \langle {n_1}, {n_2} \rangle }{1 + \frac{1}{n} },$$
    where ${n_1} :=(p-q_1)/\|p-q_1\|$, ${n_2} :=(p-q_2)/\|p-q_2\|$.
\end{lemma}
\begin{remark} Let $\beta$ is the semi-angle formed by $q_1$, $p$ and $q_2$ (see Figure \ref{fig: gradient_generalise}). From the proof, we see that there is another equivalent formulation that might be convenient in practice
    $$ \| \nabla \d_K(p) \|^2 \geq  1- \frac{2}{1 + \frac{1}{n}} \, \sin^2(\beta).$$
\end{remark}
\begin{proof}
    According to Proposition \ref{prop: expression_grad_g_rayon}, we have 
    $$\| \nabla \d_K (p) \| = \sqrt{ 1 - \bigg( \frac{\mathscr{R}_K(p)}{\d_K(p)} \bigg)^2} .$$
    By Jung's Theorem~\ref{theo: Jung} and the fact $q_1$ and $q_2$ realize the diameter of $\Gamma_K(p)$, we get
    $$ \mathscr{R}_K(p) = \mathrm{Radius}(\Gamma_K(p)) \leq \frac{1}{\sqrt{2\big( 1 + \frac{1}{n}\big)}}  \mathrm{Diam}\big( \Gamma_K(p) \big)
    = \frac{1}{\sqrt{2\big( 1 + \frac{1}{n}\big)}}\|q_1-q_2\|
    ,$$
    thus
    $$\| \nabla \d_K (p) \|^2 \geq 1 - \frac{\|q_1 - q_2\|^2}{2\big( 1 + \frac{1}{n}\big) ~ \d_K(p)^2 } .$$
        Since $\d_K(p)=\|p-q_1\|=\|p-q_2\|$, we get
    $$\|q_1 - q_2\|^2 = \|(q_1 - p) - (q_2 - p)\|^2 = 2 \d_K(p)^2 \big( 1 - \langle {n_1}, {n_2} \rangle \big),$$
and thus
$$\| \nabla \d_K (p) \|^2 \geq 1 - \frac{ 1 - \langle {n_1}, {n_2} \rangle }{1 + \frac{1}{n} }.$$
\end{proof}

\begin{proof}[Proof of Theorem \ref{theo: C1_implique_continuite_fonction_crit}]
By contradiction, we suppose that there exists $\mu < 1$ such that $\reach_\mu(M) > 0$. 
    Thus, there exists a sequence of point $(p_k)_k$ in the medial axis of $M$ such that 
    \[ \d_M(p_k) \underset{k \to \infty}{\longrightarrow} 0 ~~\text{and}~~  \|\nabla \d_M(p_k)\| \leq \mu.\] 
    By compactness of the offset of $M$, we can assume that the sequence $(p_k)_k$ converges toward a limit $p \in M$.
    For any $k$, we denote by $q_k^1$ and $q_k^2$ two projections of $p_k$ that realize the diameter of $\Gamma_M(p_k)$ and define three unitary vectors:
     \[ {n_k^1} := \frac{p_k - q_k^1}{\| p_k - q_k^1 \|}, ~~  {n_k^2} := \frac{p_k - q_k^2}{\| p_k - q_k^2 \|} ~~ \text{and} ~~ {t_k} := \frac{q_k^2 - q_k^1}{\| q_k^2 - q_k^1 \|}.\]
    Up to a subsequence we can assume that the sequence of unitary vectors ${n_k^1}$, ${n_k^2}$ and ${t_k}$ converge. We denote by ${n_1}$, ${n_2}$ and ${t}$ their respective limits.
    Using Lemma \ref{lem_normale_distnace_min_variete}, we know that ${n_k^1}$ is normal to $M$ at the point $q_k^1$, thus by continuity of the application $x \mapsto T_x M$ and of the scalar product, we get that ${n_1}$ is normal to $M$ at $p$. Similarly, we have ${n_2}$ is normal to $M$ at $p$. Since $M$ is of class $C^1$, the limit vector $t$ belongs to the tangent plane $T_pM$.
    
    The end of the proof proceeds as follows: 
    in Step 1, we show that the three vectors ${n_1}$, ${n_2}$ and ${t}$ are all contained in a plane $\Pi$. In Step 2, we show that $n_1$ and $n_2$ are not collinear. Therefore, since $n_1$ and $n_2$ are orthogonal to $T_pM$, the three vectors $n_1,n_2,t$ are independent which contradicts Step 1. \\
   
     \noindent \textbf{Step 1.} We prove here that ${n_1}$, ${n_2}$ and ${t}$ are all contained in a plane $\Pi$.
     Firstly, we remark that for $k$ large enough $p_k$, $q_k^1$ and $q_k^2$ are not  aligned. 
     Otherwise $ {n_k^1} = {t_k}$, which would imply at the limit that ${n_1} = {t}$, and
      since ${t} \in T_p M$, we get that ${n_1} \in T_p M \cap N_p M = \{0\}$. This is absurd since $\| {n_1} \| = 1$. This allows us to define, for $k$ large enough, the unique plane $\Pi_k$ containing $p_k$, $q_k^1$ and $q_k^2$ (see Figure \ref{fig:preuve_continuite_fonction_critique_contruction_plan_pi_k}).
     \begin{figure}[!ht]
    \centering
    \includegraphics[width=0.6\linewidth]{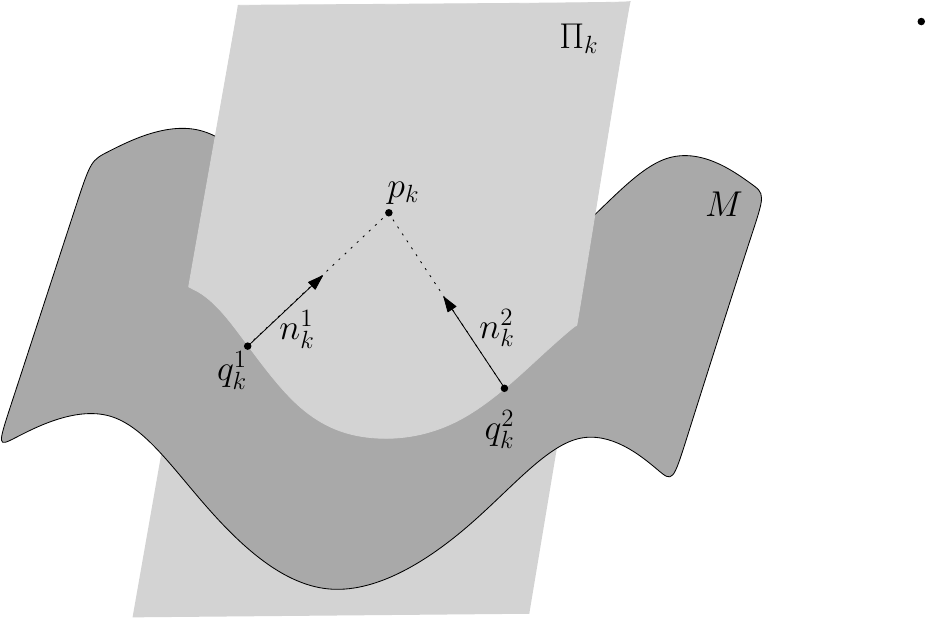}
    \caption{Proof of Theorem~\ref{theo: C1_implique_continuite_fonction_crit}.}
    \label{fig:preuve_continuite_fonction_critique_contruction_plan_pi_k}
\end{figure}
     We now define $A_M$ as the set of all the affine planes passing through a point of $M$ and let
     \[ \begin{array}{ccccc}
        f & : & M \times \mathrm{Gr}_2(\E^n) & \longrightarrow & A_M  \\
          & & (x,P) & \longmapsto & x + P \\
    \end{array}\]
    where $ \mathrm{Gr}_2(\E^n)$ is the grassmanian space of $2$-planes in  $\E^n$ .
    The map $f$ is surjective and continuous and since  $M \times \mathrm{Gr}_2(\E^n)$ is compact, we get that $A_M$ is compact. 
    Thus, since the $\Pi_k$'s are in $A_M$, up to a subsequence the sequence of planes $(\Pi_k)_k$ converges to a plane $\Pi \in A_M$ that contains ${n_1}$, ${n_2}$ and ${t}$. \\

    \noindent \textbf{Step 2.} 
    By Lemma \ref{lem: consequnce_Jung}, we have for every $k$ 
    $$\mu^2 \geq \| \nabla \d_M (p_k) \|^2 \geq 1 - \frac{ 1 - \langle{n_k^1},{n_k^2} \rangle }{1 + \frac{1}{n} } = \frac{\langle{n_k^1},{n_k^2} \rangle + \frac{1}{n} }{1 + \frac{1}{n}},$$
which implies that
    \[
    \sca{{n_k^1}}{{n_k^2}} \leq \left( 1+\frac{1}{n}\right) \mu^2 - \frac{1}{n} \leq \mu^2.
    \]
   When $k$ tends to infinity, we get $\sca{n_1}{n_2} \leq  \mu^2 <1$, hence $n_1\neq n_2$.
   
    Suppose now that $n_1=-n_2$. This would imply that for $k$ large enough, the three points $q_k^1$, $p_k$ and $q_k^2$ are almost aligned in this order. Therefore, we would have $n_1=-n_2=t$ at the limit. Hence ${t} \neq 0$ would be both tangent and normal to $M$ at $p$, which is impossible. We deduce that ${n_1}$ and ${n_2}$ are not collinear, which contradicts Step 1. 
\end{proof}

\subsection{Proof of Corollary~\ref{coro:offsetshavepositivereach}}

The proof of Corollary~\ref{coro:offsetshavepositivereach} is a direct consequence of two results in~\cite{fu1985tubular} and~\cite{chazal2007shape}. It was first proved in~\cite[Theorem 4.1]{fu1985tubular} that if $r$ is not a \textit{critical value} of the distance function $\d_M$, then the closure  $L=\overline{\E^n\setminus M^r}$ of the complement of the $r$-offset of $M$ has positive reach and its boundary $\partial L=\partial M^r$ is a Lipschitz manifold. In~\cite{fu1985tubular}, $r$ is a critical value means that the generalized gradient $\nabla\d_M$ vanishes on the level set $\d_M^{-1}(r)$, which is obviously not the case in our setting by the positive $\mu$-reach assumption. The bound on the reach was made explicit in~\cite[Theorem 4.1]{chazal2007shape} for sets with positive $\mu$-reach and shown to be more than $\mu r$. 
The second point of the corollary is a consequence of Federer's result that the level sets with positive reach are of class $C^{1,1}$.

\section{Critical function of $C^{1,\alpha}$  submanifolds}\label{sec:theoremC1alpha}
Theorem \ref{theo: C1_implique_continuite_fonction_crit} can be recast with the critical function as
\[\underset{d \to 0}{\lim} \, \chi_M(d) = 1.\]
Hence, $\chi_M$ can be extended by continuity at 0 by putting $\chi_M(0)=1$.
The critical function  appears as a natural object to measure the speed of convergence of the generalised gradient when approaching the submanifold. 
The goal of this section is to prove Theorem~\ref{theo: majoration_fonction_critique} which states that if
 $M$ is a $C^{1,\alpha}$ compact submanifold of $\E^n$ without boundary with $\alpha \in [0,1]$, 
then there exists a constant $C > 0$ such that, for all $p \in \E^n \setminus M$,
$$  1 - \| \nabla \d_M(p) \|^2  \leq C ~\d_M(p)^{2 \alpha} \Big( 1 - \| \nabla \d_M(p) \|^2 \Big)^{\alpha}.$$
In particular, if $\alpha <1$, then for every $p \in \E^n \setminus M$ 
$$
 1 - \| \nabla \d_M(p) \|^2  \leq C^{\frac{1}{1-\alpha}} ~ \d_M(p)^{\frac{2 \alpha}{1- \alpha}}.
$$
This last inequality can also be rephrased using the critical function, which gives the speed of convergence of $\chi_M$ at $0$, as stated in the following corollary.
\begin{corollary} \label{coro: majoration_fonction_critique}
Let $M$ be a $C^{1,\alpha}$ compact submanifold of $\E^n$ without boundary with $\alpha \in [0,1[$.
Then there exists a constant $C > 0$ such that for all $d \in \R_+$,
$$0 \leq  1 - \chi_M(d)^2 \leq C ~ d^{\frac{2 \alpha}{1- \alpha}}.$$
In particular, if $$ \frac{1}{3} < \alpha \leq  1 $$ then $\chi_M$ is differentiable at zero and $\chi_M'(0) = 0$.
\end{corollary}

\noindent In Section \ref{sec: Deux_exemples} we will prove the sharpness of this inequality on one example.


\subsection{Basic facts of $C^{1,\alpha}$  analysis}

In this section, we revisit several classical results in analysis related to $C^{1,\alpha}$ regularity and supplement them with quantitative criteria required for the subsequent analysis. 
For completeness, proofs are included; they consist of minor to moderate variations on standard arguments.
In the following $E$ and $F$ denote Banach spaces and $U$ an open subset of $E$. 

\begin{definition}
    We say that $f : X \subset E \to F$ is $\boldsymbol{\alpha}$\textbf{-Hölder} with $\alpha \in ] 0 , 1]$, there exists a constant $C > 0$ such that for every $y,z \in X$
    \[ \| f(y) - f(z)\|_F \leq C \|y -z\|_E^\alpha. \]
\end{definition}

\begin{definition}
    We say that $f : U \to F$ is \textbf{locally} $\boldsymbol{\alpha}$\textbf{-Hölder}\footnote{For some authors this can be the definition of being $\alpha$-Hölder, for instance see the Appendix B of \cite[pp 138 to 143]{ruelle1989elements} } with $\alpha \in ] 0 , 1]$, if each $x \in U$ has a neighborhood $V$ such that $f |_V$ is $\alpha$-Hölder.
\end{definition}

\begin{definition}
    Let $g : U \to F$, we say that $g$ is of \textbf{class} $\boldsymbol{C^{1,\alpha}}$ with $\alpha \in ] 0 , 1]$, if $g$ is of class $C^1$ and if $\d g : U \mapsto \mathscr{L}(E,F)$ is locally $\alpha$-Hölder.
\end{definition}

\noindent  The following result can be extracted from \cite{ANDERSSON1997} except the control of the constant $C$ and $C'$.
\begin{lemma} \label{lem: carac_c_1_alpha}
    Let $g : U \to F$ be a $C^{1}$ function, $\alpha \in [0,1]$ and $U$ a convex open set.
    The following statements are equivalent
    \begin{enumerate}
        \item there exists $C > 0$ such that for all $x,y \in U$,
    \[\opnorm{ \d f_x - \d f_y} \leq C \|x -y\|_E^\alpha,\]
    where $\opnorm{\cdot}$ denotes the operator norm.

        \item there exists $C' > 0$ such that for all $x,y \in U$,
    \[\| f(y) - f(x) - \d f_x(y-x)\|_F \leq C' \|y-x\|_E^{1 + \alpha}.\]
    \end{enumerate}
    Moreover, we have $(1+\alpha)\,  C \leq C' \leq 6 C$. Furthermore, the implication $(2) \implies (1)$ does not require the convex hypothesis on $U$.
\end{lemma}

\begin{proof}
    We assume that there exists $C > 0$ such that for all $x,y \in U$,
    \[ \opnorm{\d f_x - \d f_y} \leq C \|y -z\|_E^\alpha.\]
    Since, $U$ is convex, we have
    \begin{align*}
        f(y) - f(x) - \d f_x(y-x) &= \int_0^1 \frac{\d}{\d t} \Big( f\big(x + (y-x)t\big) \Big) \d t -  \d f_x(y-x) \\&=  \int_0^1  \d f_{x + t(y-x) }(y-x) -  \d f_z(y-x) \d t.
    \end{align*}
    Thus 
    \[ \| f(y) - f(x) - \d f_x(y-x)\|_F \leq C \int_0^1 \| t (y-x) \|^\alpha_E \| y-x\|_E \, \d t = \frac{C}{1+\alpha} \|y - x\|_E^{1 + \alpha}. \]
    For the converse implication, we now suppose there exists $C' > 0$ such that for all $x ,y \in U$,
    \[\| f(y) - f(x) - \d f_x(y-x)\|_F \leq C' \|y-x\|_E^{1 + \alpha}.\]
    We also have
    \[\| f(x) - f(y) - \d f_y(x-y)\|_F \leq C' \|y-x\|_E^{1 + \alpha}.\]
    thus using triangular inequality, we get
     \[\| (\d f_x - \d f_y)(x-y)\|_F \leq 2 C' \|y -x\|_E^{1 + \alpha}.\]
     Let $h$ be a unit vectors such that
     \[ \opnorm{ \d f_x - \d f_y} = \| (\d f_x - \d f_y)(h)\|. \]
     Since $U$ is open, there exist $z \in U$ such that  $h := \frac{z-x}{\|z-x\|}$ and $\|z-x\| \leq \| y - x \|$.
     We have
     \begin{align*}
         (\d f_x - \d f_y)(x-z) &= \d f_x(x-z) - \d f_y(x-y) + \d f_y(z-y) \\
         &= \big( f(z) - f(x) - \d f_x(z-x) \big) + \big( f(x) - f(y) - \d f_y(x-y) \big) \\
         &~~~~- \big( f(z) - f(y) - \d f_y(z-y) \big)
     \end{align*}
     thus using triangular inequality three times
     \[ \| (\d f_x - \d f_y)(x-z) \|_F \leq C' \|z-x\|^{1+\alpha} + C'\|x-y\|^{1+\alpha} + C' \| z-y\|^{1+\alpha}.  \]
     Since the function $x \in \R_+ \mapsto x^{1+\alpha}$ is convex, we get 
     $$ \|\frac{1}{2} (z-y) \|^{1+\alpha} \leq \frac{1}{2} \|z-x\|^{1+\alpha} + \frac{1}{2} \|x-y\|^{1+\alpha}.$$
     Thus
     \[ \| (\d f_x - \d f_y)(x-z) \|_F \leq (1 + 2^{\alpha}) C' \|z-x\|^{1+\alpha} +   (1 + 2^{\alpha}) C' \|x-y\|^{1+\alpha} \leq 2 (1 + 2^{\alpha}) C' \|x-y\|^{1+\alpha} \]
     that's why
     \[\opnorm{ \d f_x - \d f_y} \leq  2 (1 + 2^{\alpha}) C' \|x -y\|_E^\alpha \leq  6 C' \|x -y\|_E^\alpha.\]
\end{proof}


\noindent For any $C^1$ submanifold $M$ of $\E^n$, we denote by $g$ the metric on $M$ induced by the inner product of $\E^n$. We also denote by $\d_{g}$ the geodesic distance induced by $g$.

\begin{lemma} \label{lem: distance_geodesique_norme_R_n}
Let $M\subset \E^n$ be a compact $C^1$ submanifold with or without boundary then there exists $\lambda_M >0$ such that
\[\forall (x,y)\in M\times M,\quad \d_\text{g}(x,y) \leq \lambda_M \| x - y \|.\]
\end{lemma}

\begin{proof}
We denote by $m$ the dimension of $M$. Let $x \in M$. 
If $x$ lies on the interior of $M$ there exist $V_x$ an open set of $\R^m$, $\varepsilon_x > 0$ and a $C^1$-diffeomorphism $\varphi_x : M \cap B(x,\varepsilon_x) \to  V_x$.
If $x$ lies on the boundary of $M$ then $V_x$ is an open set of the half space $\R^{m-1} \times \R_+$.
In  both cases, we can further assume that $V_x$ is a convex space. 

Since $M$ is compact, we can find $\varepsilon >0$ and an integer $k \in \mathbb{N}$ such that a number $k$ of open sets of the form $M \cap B(x_i,\frac{\varepsilon}{2})$ cover $M$ (be aware that the radius is $\frac{\varepsilon}{2}$). We then denote by $\varphi_i$ the $C^1$-diffeomorphism $\varphi_{x_i}$ as above.
We then define 
$$ C_1 = \underset{i \in \{ 1 , \cdots, k \}}{\max} \opnorm{\d \varphi_i} ~~ \text{and} ~~  C_2 = \underset{i \in \{ 1 , \cdots, k \}}{\max} \opnorm{\d \varphi^{-1}_i}.$$
Let $x,y \in M$. We distinguish two cases depending on whether $\|x - y\| \leq \frac{\epsilon}{2}$ or not.\\
If $\|x - y\| \leq \frac{\epsilon}{2}$ then there exists $i \in \{ 1, \cdots, k \}$ such that $x,y \in M \cap B(x_i,\varepsilon)$. Indeed, there exists a ball $B(x_i,\varepsilon/2)$ that contains $x$ and since $\|x - y\| \leq \frac{\epsilon}{2}$, the ball $B(x_i,\varepsilon)$ also contains $y$.\\

We then define a path $\gamma:[0,1]\to M$ joining $x$ and $y$ by 
\[\gamma(t)=\varphi_i^{-1} \big( (1-t) x' + t y' \big)\]
where $x' = \varphi_i(x) ~~ \text{and} ~~ y' = \varphi_i(y).$ Since $V_i$ is convex, the path is well defined. We have
\begin{align*}
    \d_{\text{g}}(x,y) &\leq \int_0^1 \|\gamma'(t)\| \d t = \int_0^1 \| (\d \varphi_i^{-1})_{ (1-t) x' + t y'}.(y'-x') \|  \d t \leq C_2  \| y'-x' \|.
\end{align*}
By Mean Value Inequality we also have
\[ \| y'-x' \| = \| \varphi_i(x) -  \varphi_i(y)\| \leq C_1 \| y - x \|  \]
Thus, if $\|x -y\| \leq \frac{\epsilon}{2}$, we obtain
$$ \d_{\text{g}}(x,y) \leq C_1 C_2 \|x - y\|.$$
In the other case where $\|x -y\| > \frac{\epsilon}{2}$, we obviously have
\[ \d_{\text{g}}(x,y)\leq \mathrm{Diam}_g(M) \leq \frac{2 \, \mathrm{Diam}_g(M)}{\varepsilon}\|x - y\|\]
where $\mathrm{Diam}_g(M) := \underset{x,y\in M}{\max} \d_{\text{g}}(x,y)$ is finite since $M$ is compact.
To conclude, the constant
$$\lambda_M = \max \Big( C_1 C_2, \frac{2 \, \mathrm{Diam}_g(M)}{\varepsilon} \Big)$$
fits for the two cases.
\end{proof}


The next result states the implication $(1) \implies (2)$ of Lemma \ref{lem: carac_c_1_alpha} in the case where $U$ is not necessarily convex.

\begin{lemma} \label{lem: carac_c_1_alpha_2}
    Let $U $ be a connected open set of $\E^n$ such that $M := \Bar{U}$ is a compact connected $C^1$ submanifold of $\E^n$, $\alpha \in [0,1]$ and let $f : U \to F$ be a $C^{1}$ function. 
    If there exists $C > 0$ such that for all $x,y \in U$,
    \[\opnorm{ \d f_x - \d f_y} \leq C \|x -y\|_{\E^n}^\alpha,\]
    then for all $x,y \in U$,
    \[\| f(y) - f(x) - \d f_x(y-x)\|_F \leq \frac{C. \lambda_M}{1+\alpha} \|y-x\|_{\E^n}^{1 + \alpha}\]
    with $\lambda_M$ the constant of Lemma \ref{lem: distance_geodesique_norme_R_n}.
\end{lemma}

\begin{proof}
    $M$ is a $C^1$ manifold thus its locally $C^1$-path connected but since $M$ is connected and compact, it is (globally) $C^1$-path connected.
    Thus the same goes for $U$.
    We fix $x,y \in U$ and $\gamma : [0,L] \to U$ a $C^1$ path such that $\gamma(0) = x$, $\gamma(L) = y$ and $\gamma$ is parameterized by arc length.
    We have 
    \[ f(y) - f(x) - \d f_x(y-x) = \int_0^L \d f_{\gamma(t)}.\gamma'(t) - \d f_x.\gamma'(t) dt  \]
    thus by taking the norm and using the hypothesis of $f$, we got
    \[ \| f(y) - f(x) - \d f_x(y-x)\| \leq \int_0^L \opnorm{\d f_{\gamma(t)} - \d f_x} dt \leq C \int_0^L \|\gamma(t) - x\|^\alpha \d t.\]
    However, using mean value inequality 
    \[ \int_0^L \|\gamma(t) - \gamma(0)\|^\alpha \d t \leq \int_0^L t^\alpha \d t = \frac{L^{\alpha+1}}{\alpha + 1}. \]
    Taking the lower bound on the choice of $\gamma$, we got
    \[ \| f(y) - f(x) - \d f_x(y-x)\| \leq \frac{C}{1+\alpha} \d_{g}(x,y)^{1+\alpha}\]
    therefore, using Lemma \ref{lem: distance_geodesique_norme_R_n}, we get
    \[ \| f(y) - f(x) - \d f_x(y-x)\| \leq \frac{C. \lambda_M}{1+\alpha} \| x-y \|^{1+\alpha}.\]
\end{proof}


\begin{definition}
    We say that $f \in C^{1,\alpha}$ is a $\boldsymbol{C^{1,\alpha}}$\textbf{-diffeomorphism} if $f$ has a $C^{1,\alpha}$ inverse.
\end{definition}

\noindent The Inverse Function Theorem still holds for $C^{1,\alpha}$ maps (see \cite[pp 141]{ruelle1989elements}).
We give here the classical proof since it will be used later as a guideline in the proof of Proposition \ref{prop: graph_constant_holder_global}.

\begin{theorem}[Inverse Function Theorem] \label{theo: inverse_function}
    Let $f : U \to F$ of class $C^{1,\alpha}$ with $\alpha \in [ 0 , 1]$ and let $a \in U$. We suppose that $\d f_a$ is an isomorphism of $\mathscr{L}(E,F)$. Then there exists an open set $V \subset U$ such that $a \in V$ and $f \vert_V$ is a $C^{1,\alpha}$-diffeomorphism on its image.
\end{theorem}

\begin{proof}
    Up to translations and composition by $ (\d f_a)^{-1} $, the proof reduces to the case $E = F$, $a = f(a) = 0$ and $\d f_0 = \mathrm{id}_E$.
    We set $\varphi(x) := f(x) - x$ and for all $y \in E$, $T_y(x) := y - \varphi(x)$ so that the equation $y = f(x)$ is equivalent to $T_y(x) = x.$
    Since $d\varphi_0 = 0$, we have $ \opnorm{ d\varphi_x } < \frac{1}{2}$ on a open ball $B(0,r)$. Using the Mean value inequality, we get for all $x_1,x_2 \in B(0,r)$,
   $$ \| T_y(x_1) - T_y(x_2) \| < \frac{1}{2} \| x_1 - x_2\|.$$
From the fact that
    $$\| \varphi(x) - \varphi(0) \| \leq \frac{1}{2} \| x-0 \| \leq \frac{r}{2}$$
for every $x \in \overline{B(0,r)},$ we deduce by Triangle Inequality
    $$\| T_y(x) \| < \frac{r}{2} + \|y\|.$$
Thus, for all $y \in B(0,\frac{r}{2}),$ we have $T_y \big( \overline{B(0,r)} \big) \subset \overline{B(0,r)}$.
We then apply Banach Fixed Point Theorem to the contraction $T_y |_{\overline{B(0,r)}}$ for all $y \in B(0,\frac{r}{2})$ to obtain the existence of a unique $x \in \overline{B(0,r)}$ such that $f(x) = y$. We define $V := f^{-1} \big(  B(0,\frac{r}{2}) \big)$ so that $f |_V$ is a bijection on its image and we denote by $g$ its inverse. 
Let $y_1,y_2 \in f(V)$ and $x_1,x_2 \in V$ such that $f(x_1) = y_1$ and $f(x_2) = y_2$.
    We have
    \[ g(y_2) - g(y_1) = x_2 - x_1 = T_{y_2}(x_2) - T_{y_1}(x_1) = y_2 - \varphi(x_2) - y_1 + \varphi(x_1) .\]
    Since $\varphi$ is $\frac{1}{2}$-Lipschitz, we get
    \[ \| g(y_2) - g(y_1) \| \leq \| y _2 - y_1 \| + \frac{1}{2} \| x_2 - x_1\| = \| y _2 - y_1 \| + \frac{1}{2} \| g(y_2) - g(y_1) \|. \]
This shows that the map $g$ is $2$-Lipschitz.\\ 

\noindent
Since $f$ is of class $C^{1,\alpha}$, there exists  a neighborhood $V_0\subset V$ of $0$ such that $\d f$ is $\alpha$-hölder. By shrinking this neighborhood if needed, we can assume that $V_0$ is convex. Let $x_1,x_2\in V_0$, $y_1=f(x_1),$ $y_2=f(x_2),$ 
    $$ y_2 - y_1 = \d f_{x_1}(x_2-x_1) + R_{x_1}(x_2)$$
where the remainder satisfies $R_{x_1}(x_2) = O(\|x_2-x_1\|^{1+\alpha})$, see Lemma~\ref{lem: carac_c_1_alpha}. 
 Since $\d f_0$ is invertible, $\d f_x$ is also invertible for all $x$ in some ball $B(0,r')$. We replace $r$ by $\min(r,r')$. Composing by $(\d f_{x_1})^{-1}$, we obtain
    \begin{equation} \label{eq: _inv_loc_definition_C_1_alpha}
        g(y_2) - g(y_1) = x_2-x_1 = (\d f_{x_1})^{-1}(y_2 - y_1) - (\d f_{x_1})^{-1}\big(R_{x_1}(x_2)\big).
    \end{equation}
    Since $g$ is $2$-Lipschitz, we have that $\|x_2-x_1\| \leq 2 \|y_2-y_1\|$ thus
    \begin{align*}
        \frac{\| (\d f_{x_1})^{-1}\big(R_{x_1}(x_2)\big) \|}{\|y_2-y_1\|^{1+\alpha}} &\leq 2^{1+\alpha} \frac{\| (\d f_{x_1})^{-1}\big(R_{x_1}(x_2)\big) \|}{\|x_2-x_1\|^{1+\alpha}} \\ 
        &\leq 2^{1+\alpha} \bigg\| (\d f_{x_1})^{-1}\bigg( \frac{R_{x_1}(x_2)}{\|x_2-x_1\|^{1+\alpha}} \bigg) \bigg\| = O_{x_2\to x_1}(1).
    \end{align*}
    We deduce that 
    \begin{equation} \label{eq: inverse_c_1_alpha}
        (\d f_{x_1})^{-1}\big(R_{x_1}(x_2)\big) = O(\|y_2-y_1\|^{1+\alpha}).
    \end{equation}
    Thus by Equation~\eqref{eq: _inv_loc_definition_C_1_alpha}, $g$ is differentiable and 
    $\d g_{y_1} = (\d f_{g(y_1)})^{-1}.$
    Since $g$ and $\d f$ are continuous, $g$ is of class $C^1$. Equation \eqref{eq: inverse_c_1_alpha} and Lemma \ref{lem: carac_c_1_alpha} imply that $g$ is of class~$C^{1,\alpha}$.
\end{proof}

\begin{lemma} \label{lem: constant_holder_inversion_local}
    Let $V$ be an open set of $E$ and $f : V \to F$ be a $C^{1}$-diffeomorphism on its image.
    Let $W \subset V$ be a convex neighborhood of some $a \in V$ such that $\d f |_W$ is $\alpha$-Holder with $\alpha \in [ 0 , 1]$ and such that for all $x \in W$, 
    \begin{equation} \label{eq: condition_alpha_holder_sur_tout_ouvert}
        \opnorm{(\d f_x)^{-1}} \leq 2  \opnorm{(\d f_a)^{-1}} ~~ \text{and} ~~ \opnorm{ (\d f_a)^{-1} \circ \d f_x   - id_E } \leq \frac{1}{2}  .
    \end{equation}
    Then $\d f^{-1}$ is $\alpha$-Hölder on $f(W)$ and its Hölder constant $C'$ satisfies
    $$ C' \leq \frac{ 3.2^{3 + \alpha} \, \opnorm{ (\d f_a)^{-1} }^{2+\alpha}}{1 + \alpha} C$$
    where $C$ is the Hölder constant of $\d f|_W$.
\end{lemma}

\begin{proof}
    We consider $\Tilde{f}: E \to E$ defined by 
    $$\Tilde{f} = (\d f_a)^{-1} \circ t_{-f(a)} \circ f \circ t_{a}$$
    with $t_v : x \mapsto x + v $.
    We observe that $ \Tilde{f}(0) = 0$ and $\d \Tilde{f}_0 = \mathrm{id}_E$.
    Obviously
    $$\Tilde{f}^{-1} = t_{-a} \circ f^{-1} \circ t_{f(a)} \circ \d f_a$$
    thus
    $$f^{-1} = t_a \circ \Tilde{f}^{-1} \circ (\d f_a)^{-1}  \circ  t_{-f(a)}.$$
    Now, since $f$ is of class $C^{1,\alpha}$, for all $x_1,x_2 \in W$, we have
     $$ y_2 - y_1 = \d f_{x_1}(x_2-x_1) + R_{x_1}(x_2- x_1)$$
    where $f(x_1) = y_1$ and $f(x_2) = y_2$. 
    According to Lemma \ref{lem: carac_c_1_alpha}, we have
    $$\| R_{x_1}(x_2 - x_1) \| = \| f(x_2) - f(x_1) - \d f_{x_1}(x_2-x_1)\| \leq 6 C \|x_2 -x_1\|^{1 + \alpha}.$$
    Composing by $(df_{x_1})^{-1}$, we get
    $$f^{-1}(y_2) - f^{-1}(y_1) = x_2-x_1 = (df_{x_1})^{-1}(y_2 - y_1) - (df_{x_1})^{-1}\big(R_{x_1}(x_2-x_1)\big)$$
    and thus
    \[ \| f^{-1}(y_2) - f^{-1}(y_1) - (\d f_{x_1})^{-1}(y_2 - y_1) \| = \| (\d f_{x_1})^{-1}\big(R_{x_1}(x_2-x_1)\big) \|.\]
    It remains to find an upper bound of the right term.
    We have
    \[ \| (\d f_{x_1})^{-1}\big(R_{x_1}(x_2-x_1)\big) \| \leq \opnorm{(\d f_{x_1})^{-1} } \| R(x_2-x_1) \| \leq  6 C \opnorm{(\d f_{x_1})^{-1}} \|x_2-x_1\|^{1+\alpha} .\]
    Using the first assumption of \eqref{eq: condition_alpha_holder_sur_tout_ouvert} on $W$, we obtain 
    \begin{equation} \label{eq: etape_controle_reste}
         \| (\d f_{x_1})^{-1}\big(R_{x_1}(x_2-x_1)\big) \| \leq   12 C \opnorm{(\d f_{a})^{-1}} \|x_2-x_1\|^{1+\alpha} .
    \end{equation}
    Since 
    $$\d \Tilde{f} =  (\d f_a)^{-1} \circ \d f,$$
    the second assumption of \eqref{eq: condition_alpha_holder_sur_tout_ouvert} reads, for all $x \in W$
    \[ \opnorm{ \d \Tilde{f}_x - id_E } \leq \frac{1}{2}. \]
    Now, using the same argument as in the proof of Theorem \ref{theo: inverse_function}, $g$ being $\Tilde{f}^{-1}$, we deduce that $\Tilde{f}^{-1} $ is $2$-Lipschitz and thus $f^{-1}$ is $2 \opnorm{ (\d f_a)^{-1} } $-Lipschitz on $W$. 
    In other words
    $$\|x_2-x_1\| \leq 2  \opnorm{ (\d f_a)^{-1} } \|y_2-y_1\|.$$ 
    Plugging this result in Inequality \eqref{eq: etape_controle_reste}, we have
     \[ \| (\d f_{x_1})^{-1}\big(R_{x_1}(x_2-x_1)\big) \| \leq 3. 2^{3 + \alpha} C   \opnorm{(\d f_a)^{-1}}^{2+\alpha}  \|y_2-y_1\|^{1+\alpha}.\]
     According to  Lemma \ref{lem: carac_c_1_alpha}, the $\alpha$-Hölder constant  $C'$ of $df^{-1}$ on $f(W)$ satisfies
     \[ C' \leq \frac{ 3.2^{3 + \alpha} \, \opnorm{(\d f_a)^{-1}}^{2+\alpha}}{1 + \alpha} C .\]
\end{proof}


\subsection{Submanifolds of class \texorpdfstring{$C^{1,\alpha}$}{C(1,a)}}

The goal of this section is to show that at any point $q$ of a $C^{1,\alpha}$ submanifold $M$, there exists a map $h_q$ such that $M$ is locally the graph of $h_q$. Moreover, for any $q_0$ in $M$, there exists an open set $U$ such that all the $\alpha$-Hölder constant of $d h_q$ for $q \in M \cap U$ are bounded. This will be stated in Proposition \ref {prop: graph_constant_holder_global} that will be used later in the proof of Theorem~\ref{theo: majoration_fonction_critique}.

\begin{definition}
    A subset $M$ of a Banach space $E$ is a $\boldsymbol{C^{1,\alpha}}$ \textbf{submanifold}, if for all $p \in M$ there exist an open set $\mathcal{U}$ of $E$ containing $p$ and a $C^{1,\alpha}$-diffeomorphism $\varphi : \mathcal{U} \mapsto \mathcal{V} \times \mathcal{W} \subset F \times G$ such that 
    $$\varphi(M \cap \mathcal{U}) = \mathcal{V} \times \{ 0 \}$$
    where $\mathcal{V}$ is an open subset of the Banach spaces $F$ and $\mathcal{W}$ is an open subset of the Banach spaces $G$ with $0 \in \mathcal{W}$.
\end{definition}

\begin{remark}
    Due to the Inverse Function Theorem, all of the classical local characterizations of a submanifold \textit{i.e.} as a graph, as the zeros of a submersion and as the image of a parametrization, have they $C^{1,\alpha}$ counterparts.
\end{remark}


\begin{proposition} \label{prop: graph_constant_holder_global}
    Let $M$ be a $C^{1, \alpha}$ submanifold of $\E^n$ and $q_0 \in M$.
    There exist $U \subset \E^n$ an open neighborhood of $q_0$ and $K  = K(U) > 0$ such that for all $q \in  M \cap U$
    \begin{itemize}
        \item[(1)] there is an open set $V_q \subset T_q M$ such that $M$ is locally the graph of a $C^{1, \alpha}$ function $h_{q} : V_q \subset T_{q}M \to N_{q}M$, \textit{i.e.} 
        $$ M \cap U = \big\{ \big(x , h_{q}(x)\big) ~| ~ x \in V_q  \subset T_q M\big\}, $$
        \item[(2)] $\d (h_{q})$ is $\alpha$-Hölder on $U$ and its Hölder constant is bounded by $K$.
    \end{itemize}
\end{proposition}

\begin{proof}
    For any $q \in M$, we denote by $\pi_{q}^T$ the orthogonal projection on $T_q M + q$ and  $\pi_{q}^N$ the orthogonal projection on $N_q M + q$. In order to prove the proposition, we just need to prove that
    there exist $U \subset \E^n$ an open neighborhood of $q_0$ and $K > 0$ such that 
    \begin{itemize}
        \item[(i)] for all $q \in  M \cap U$, $\pi_{q}^T |_{M\cap U}$ is a $C^{1,\alpha}$-diffeomorphism on its image,
        \item[(ii)] for all $q \in  M \cap U$, $\d \big( \pi_{q}^T |_{M\cap U} \big)^{-1}$ is $\alpha$-Hölder on $M \cap U$ and its Hölder constant is bounded by $K$.
    \end{itemize}
    Indeed, if we set $h_{q} = \pi_{q}^N \circ \big( \pi_{q}^T \big)^{-1}$ we obtain exactly $(1)$ and $V_q  = \pi_{q}^T(M\cap U)$.
    Then, since $\| \d (\pi^N_{q})\| \leq 1$ and by the chain rule, we have that $\d (h_{q})$ is $\alpha$-Hölder on $V_{q}$ with its Hölder constant bounded by $K$, hence $(2)$ is proved.
    
    The proof of $(i)$ and $(ii)$ is done in the next five steps.
    The first four steps will prove $(i)$ by mimicking the proof of the Theorem \ref{theo: inverse_function} and adding a parameter $q\in M.$
    We fix $q_0 \in M$ and $\psi : V_0 \subset \R^m \to U_0 \cap M$ a $C^{1,\alpha}$ parametrization of $M$ such that $\psi(0) = q_0$ where $U_0$ (resp. $V_0$) is an open sets of $\E^n$ (resp. $\R^m$).
    We also suppose $U_0$ small enough so that $\d \psi$ and $\d (\psi^{-1})$ are $\alpha$-Hölder.
    We denote by $C_{\psi}$ and $C_{\psi^{-1}}$ their Hölder constant on $U_0$ (resp. $V_0$).
    We define, for all $q \in U_0$
    \[ \begin{array}{ccccc}
    f_q & : & V_0 \subset \R^m & \longrightarrow & \R^m \\
     & & x & \longmapsto & \d (\psi^{-1})_{q} \circ t_{-q}  \circ \pi_{q}^T \circ \psi(x) \\
    \end{array}\]
    with $t_{-q} : x \mapsto x - q$.
    We also define $\varphi_{q} = f_{q} - id$ and
    for all $y \in \R^m$, we define $T_{y,q}(x) = y -\varphi_{q}(x)$. 
    We then get,
    $$ y = f_{q}(x) \iff T_{y,q}(x) = x.$$
    In order to show that $\pi_q^T |_M$ is locally a bijection, we need to show that $f_q$ is a bijection, which is recast as a fixed point problem for $T_{y,q}$\\
    
    \noindent \textbf{Step 1.} The goal of this step is to show that there exists $r > 0$ such that $T_{y,q}$ is $\frac{1}{2}$-Lipchitz on $B(0,r)$, for every $y \in \R^m$ and $q \in \psi \big( B(0,r) \big)$.    
    Using chain rules, we have for $x \in V_0$,
    \[ \d (\varphi_{q})_{x} = \d ( f_{q})_{x} - id = \d (\psi^{-1})_{q} \circ \d \big( \pi_{q}^T\big) \circ \d \psi_{x} - id .  \]
    Also, since $ \pi_{q}^T(q) = q$ and $\d (\pi_{q}^T)_{q} |_{T_q M} = id_{T_{q}M}$, we have
    $$\d (f_{q})_{\psi^{-1}(q)} = \d (\psi^{-1})_{q}  \circ \d (t_{-q}) \circ \d \big( \pi_{q}^T\big) \circ \d \psi_{\psi^{-1}(q)} = \d (\psi^{-1})_{q} \circ \d \psi_{\psi^{-1}(q)} = id_{\R^m}$$
    thus $\| \d (\varphi_{q_0})_{0} \| = 0$.
    Moreover, we know that  and that the application 
    \[ (x,q) \in V_0 \times M \cap U_0 \mapsto  \| \d (\psi^{-1})_{q} \circ \d \big( \pi_{q}^T\big) \circ \d \psi_{x} - id  \| \]
    is continuous.
    Thus, there exists $r > 0$ such that $B(0,r) \subset V_0$ and for all $(x, q) \in B(0,r) \times \psi \big( B(0,r) \big)$, we have
    \begin{equation} \label{eq: maitrise_condition_projection}
        \opnorm{ \d ( \varphi_{q})_{x} }  \leq \frac{1}{2}.
    \end{equation}
    Using the Mean value inequality, we get that for all $q \in \psi \big( B(0,r) \big)$ and $ x_1,x_2 \in  B(0,r)$,
    $$ \| \varphi_{q}(x_1) - \varphi_{q}(x_2) \| < \frac{1}{2} \| x_1 - x_2\|$$
    thus $$ \|T_{y,q}(x_1) -T_{y,q}(x_2) \| < \frac{1}{2} \| x_1 - x_2\|.$$

    \noindent \textbf{Step 2.}
    There exists $r' > 0$ such that 
    $$M \cap B(q_0,r') \subset \psi \big( B(0,\frac{r}{8}) \big)$$ 
    and we define
    $$U := B\Big(q_0, \frac{r'}{2}\Big).$$
    We fix $q \in U \cap M$. 
    Our goal is to prove that, up to a restriction to a well chose set, $\pi^T_{q}$ is a bijection on its image.
    We have that 
    $$U \cap M \subset M \cap B(q_0,r') \subset \psi \big( B(0,\frac{r}{8}) \big) \subset \psi \big( B(0,r) \big)$$
    thus for $q \in U \cap M$, $T_{y,q}$ is $\frac{1}{2}$-Lipchitz on $ B(0,r)$.
    Thus for $T_{y,q}$ to be a contraction, we just need to prove that $T_{y,q} \big( \overline{B(0,\frac{r}{2})} \big) \subset  B(0,\frac{r}{2})$.
    We now suppose that $y \in B(0,\frac{r}{8})$, then for all $x \in \overline{B(0,\frac{r}{2})}$,we have
    \begin{align*}
        \| \varphi_{q}(x) \| &= \| \varphi_{q}(x) - \varphi_{q}\big(\psi^{-1} (q)\big) \| \leq \frac{1}{2} \| x-\psi^{-1}(q) \| \\
        &\leq \frac{1}{2} \| x \| + \frac{1}{2} \| \psi^{-1}(q) \| < \frac{r}{4} + \frac{r}{8} \leq \frac{3r}{8}
    \end{align*}
    because $q \in U \cap M \subset \psi \big( B(0,\frac{r}{8}) \big)$.
    Thus, we get
    $$\|T_{y,q}(x) \| < \frac{3r}{8} + \|y \| <  \frac{3r}{8} + \frac{r}{8} \leq \frac{r}{2}.$$
    Since $  \overline{B(0,\frac{r}{2})}$ is a complete metric space, we can apply Banach fixed point theorem of $T_{y,q} |_{\overline{B(0,\frac{r}{2})}}$ for all $y \in B(0,\frac{r}{8})$.
    Hence, for all $y \in B(0,\frac{r}{8})$ there exists a unique $x \in B(0,\frac{r}{2})$ such that $f_q(x)=y$.
    Therefore, if we define $U_{q'} \subset \R^m$ such that 
    $$U_{q} = \big(f_{q}\big)^{-1} \big( B(0,\frac{r}{8}) \big)$$
    then $f_{q} |_{U_{q}}$ is a bijection on its image (being $B(0,\frac{r}{8})$).
    Since $\d (\psi^{-1})_q$ and $\psi$ are bijections, we have that 
    $\pi^T_{q} |_{\psi(U_{q})}$ is a bijection on its image for all $q \in U \cap M$.\\

    \noindent \textbf{Step 3.} The goal of this step is to prove $(i)$.
    We first show that, up to a smaller choice of $r'>0$, we have that for all $q \in M \cap B(q_0,r')$
    \begin{equation} \label{eq: inclusion_differentielle_parametrisation}
        T_{q} M \cap B(0,r') \subset  \d \psi_{\psi^{-1}(q)} \Big( B\big(0,\frac{r}{8}\big) \Big).
    \end{equation}
    Indeed, thanks to the continuity of $x \mapsto  \d \psi_{\psi^{-1}(x)}$, the function
    \[ \lambda : q \in  \psi\big(B\big(0,\frac{r}{8}\big)\big) \mapsto \underset{\|v\| = 1}{\sup} \|  \d \psi_{\psi^{-1}(q)} .v  \|\]
    is continuous. 
    Since it is defined on a compact set, it has a minimum $\lambda_0$.
    Then, we have for all $q \in \psi\big(B (0,\frac{r}{8})\big)$
    \[   T_{q} M \cap B(0,\frac{\lambda_0 r }{8}) \subset  T_{q} M \cap B(0,\frac{\lambda(q) \, r }{8}) \subset \d \psi_{\psi^{-1}(q)} \Big( B\big(0,\frac{r}{8}\big) \Big).\]
    Therefore, by taking $r'$ to be the minimum between the previous $r'$ and $\frac{\lambda_0 r }{8}$ we get \eqref{eq: inclusion_differentielle_parametrisation}.
    Now, we prove that
    \[ U \cap M \subset \underset{q \in  M \cap B(q_0,\frac{r'}{4})}{\bigcap} \psi(U_{q}).\]
    Let $q \in B\big(q_0, \frac{r'}{2}\big)$, since $\pi_q^T$ retracts distances, we have that
    $$  \pi^T_{q} \Big(  B\Big(q_0, \frac{r'}{2}\Big) \Big) \subset  B\Big( \pi_q^T(q_0), \frac{r'}{2}\Big)$$
    and
    $$ \| \pi_q^T (q_0) - q \| \leq \| q_0 - q \| \leq \frac{r'}{2} $$
    thus
    $$   t_{-q} \circ \pi^T_{q} \Big(  B\Big(q_0, \frac{r'}{2}\Big) \Big) \subset  B( 0 , r').$$
    Therefore,
    $$   t_{-q} \circ \pi^T_{q} ( U \cap M )  \subset B(0 , r') \cap T_q M.$$
    Thanks to Equation \eqref{eq: inclusion_differentielle_parametrisation}, we have
    \[   t_{-q} \circ \pi^T_{q} (U \cap M)  \subset \d \psi_{\psi^{-1}(q)} \big( B(0,\frac{r}{8}) \big) .\]
    Hence by the definition of $f_q$, we have than $ \big( t_q \circ \pi_q^T) ^{-1} \circ \d \psi_{\psi^{-1}(q)} = \psi \, \circ (f_q)^{-1}$, thus
    \[U \cap M \subset \psi \circ \big( f_{q} \big)^{-1} \Big( B\big(0,\frac{r}{8}\big) \Big)  = \psi (U_{q}).\]
    Now, thanks to Step 2, we have for all $q \in U \cap M$, $\pi_{q}^T |_U$ is a bijection on its image $V_{q}$. \\

    \noindent \textbf{Step 4.}
    In this step, similarly as in the proof of Theorem \ref{theo: inverse_function}, we show that all the $\pi_{q}^T |_{U \cap M}$ are $C^{1,\alpha}$-diffeomorphisms on their respective image.
    
    \noindent We fix $q \in U \cap M$.
    Firstly, we will show that $f_q^{-1}$ is continuous. Let $y_1,y_2 \in B(0,\frac{r}{8})$. Then there exist $x_1,x_2 \in U_q$ such that $f_q(x_1) = y_1$ and $f_q(x_2) = y_2$.
    We have
    \[ f_q^{-1}(y_2) - f_q^{-1}(y_1) = x_2 - x_1 = T_{y_2,q}(x_2) - T_{y_1,q}(x_1) = y_2 - \varphi_q(x_2) - y_1 + \varphi_q(x_1). \]
    Since $\varphi_q$ is $\frac{1}{2}$-Lipschitz, we get
    \[ \| f_q^{-1}(y_2) - f_q^{-1}(y_1) \| \leq \| y _2 - y_1 \| + \frac{1}{2} \| x_2 - x_1\| = \| y _2 - y_1 \| + \frac{1}{2} \| f_q^{-1}(y_2) - f_q^{-1}(y_1) \| \]
    thus
    \begin{equation*} 
        \| f_q^{-1}(y_2) - f_q^{-1}(y_1) \| \leq 2 \| y _2 - y_1 \|.
    \end{equation*}
    This shows that the map $f_q^{-1}$ is $2$-Lipschitz and thus continuous. \\

    \noindent Secondly, we are going to show that $f_q^{-1}$ is of class $C^{1,\alpha}$ using Lemma \ref{lem: carac_c_1_alpha}.
    Since $\d (f_{q_0})_0$ is invertible and $(q,x) \mapsto \d (f_q)_x$ is continuous, we can suppose $r$ small enough so that for all $q \in U$ and $x \in U_q$, $\d (f_q)_x$ is invertible.
    We recall that  $\d \psi$ and $\d (\psi^{-1})$ are $\alpha$-Hölder on $U \cap M$ and $\psi(U \cap M)$.
    For all $q \in U \cap M$, $d \pi_q^T$ is 1-Lipchitz thus $\d (f_{q})$ is $\alpha$-Hölder on $W :=\psi^{-1}(U \cap M)$.
    By shrinking $U$ if needed, we can assume that $W$ is convex. 
    We fix $q \in U\cap M$, let $x_1,x_2\in W$, $y_1=f_q(x_1),$ $y_2=f_q(x_2),$ 
    $$ y_2 - y_1 = \d (f_q)_{x_1}(x_2-x_1) + R_{x_1}(x_2)$$
    where the remainder satisfies $R_{x_1}(x_2) = O(\|x_2-x_1\|^{1+\alpha})$, see Lemma~\ref{lem: carac_c_1_alpha}. 
    Composing by $\big(\d (f_q)_{x_1}\big)^{-1}$, we obtain
    \begin{equation} \label{eq: _inv_loc_definition_C_1_alpha_2}
        (f_q)^{-1}(y_2) -  (f_q)^{-1}(y_1) = x_2-x_1 = \big(\d (f_q)_{x_1}\big)^{-1}(y_2 - y_1) - \big(\d (f_q)_{x_1}\big)^{-1}\big(R_{x_1}(x_2)\big).
    \end{equation}
    Since $(f_q)^{-1}$ is $2$-Lipschitz, we have that $\|x_2-x_1\| \leq 2 \|y_2-y_1\|$ thus
    \begin{align*}
        \frac{\| \big(\d (f_q)_{x_1}\big)^{-1}\big(R_{x_1}(x_2)\big) \|}{\|y_2-y_1\|^{1+\alpha}} &\leq 2^{1+\alpha} \frac{\| \big(\d (f_q)_{x_1}\big)^{-1}\big(R_{x_1}(x_2)\big) \|}{\|x_2-x_1\|^{1+\alpha}} \\ 
        &\leq 2^{1+\alpha} \bigg\| \big(\d (f_q)_{x_1}\big)^{-1}\bigg( \frac{R_{x_1}(x_2)}{\|x_2-x_1\|^{1+\alpha}} \bigg) \bigg\| = O_{x_2\to x_1}(1).
    \end{align*}
    We deduce that 
    \begin{equation} \label{eq: inverse_c_1_alpha_2}
        \big(\d (f_q)_{x_1}\big)^{-1}\big(R_{x_1}(x_2)\big) = O(\|y_2-y_1\|^{1+\alpha}).
    \end{equation}
    Thus by Equation \eqref{eq: _inv_loc_definition_C_1_alpha_2}, $f_q^{-1}$ is differentiable and 
    $$\d \big( (f_q)^{-1} \big)_{y_1} = \big(\d (f_q)_{(f_q)^{-1}(y_1)} \big)^{-1}.$$
    Since $f_q^{-1}$ and $\d (f_q)$ are continuous, $f_q^{-1}$ is of class $C^1$. Equation \eqref{eq: inverse_c_1_alpha_2} and Lemma \ref{lem: carac_c_1_alpha} imply that  $f_q^{-1}$ is of class $C^{1,\alpha}$.
    For all $q \in U$, we recall that
    $$(\pi_{q}^T)^{-1} |_{V_q} = \psi \circ f_{q}^{-1} \circ \d \big(\psi^{-1}\big) _q \circ t_{-q} |_{V_q} $$
    where $V_q = \pi_{q}^T (U \cap M)$.
    Thus the $\pi_{q}^T |_{U \cap M}$ are $C^{1,\alpha}$-diffeomorphisms on their respective image.\\

    \noindent \textbf{Step 5.} 
    In this last step, we will prove the point $(ii)$ of the proposition.
    We denote by $C_{\psi}$ and $C_{\psi^{-1}}$ the $\alpha$-Hölder of constant of $\d \psi$ and $\d (\psi^{-1})$ on $U \cap M$ and $\psi(U \cap M)$.
    For all $q \in U \cap M$, since $d \pi_q^T$ is 1-Lipchitz, we have that $\d(f_{q})$ is $\alpha$-Hölder on $W = \psi^{-1}(U \cap M)$, we denote its constant by $C_{f_{q}}$.
    We have 
    \[ C_{f_{q}} \leq \opnorm{\d \psi_{q}} \, C_{\psi} \leq \opnorm{(\d \psi)_{|U}  }\, C_{\psi}\]
    where the right term is finite since $U$ is bounded.
    Moreover, for all $q\in U \cap M$, since $\d (f_{q})_{\psi^{-1}(q)} = id_{\R^m}$ and using Equation \eqref{eq: maitrise_condition_projection}, we get for all $x \in W$
    \[ \opnorm{\d (f_q)_x - id_{\R^m} } = \opnorm{ \d ( \varphi_q)_x } \leq \frac{1}{2}\]
    which is exactly the first inequality of Condition \eqref{eq: condition_alpha_holder_sur_tout_ouvert} of Lemma \ref{lem: constant_holder_inversion_local} at the point $\psi^{-1}(q)$ and for the open set $W$.
    Moreover, we have
    \[ \opnorm{ \big( \d (f_q)_x \big)^{-1} } -1 \leq \opnorm{  \big( \d (f_q)_x \big)^{-1} - id_{\R^m} } \leq \opnorm{ \big( \d (f_q)_x \big)^{-1} }  \opnorm{ \d (f_q)_x - id_{\R^m}  } \]
    thus using the two previous inequality, we have
    \[ \opnorm{ \big( \d (f_q)_x \big)^{-1} } -1 \leq \frac{1}{2} \opnorm{ \big( \d (f_q)_x \big)^{-1} } \]
    which implies that for all $x \in W$,
    \[ \opnorm{ \big( \d (f_q)_x \big)^{-1} } \leq 2\]
    which is the second inequality of Condition \eqref{eq: condition_alpha_holder_sur_tout_ouvert}.
    Thus we can apply Lemma \ref{lem: constant_holder_inversion_local} to $f_{q}$ at the point $\psi^{-1}(q)$ and for the convex open set $W$.
    Thus $\d (f_{q})^{-1}$ is $\alpha$-Hölder on $f_{q}(W') = \d \psi_{q}(V_{q} - q)$ and its constant $C_{f^{-1}_{q}}$ verifies
    \[ C_{f^{-1}_{q}} \leq \frac{3.2^{3+\alpha} \opnorm{ \big(\d (f_{q})_{\psi^{-1}(q)} \big)^{-1} }}{1 + \alpha} C_{f_{q}} \leq \frac{3.2^{3+\alpha}}{1 + \alpha} \, {\opnorm{(\d \psi)_{|U} }} \, C_{\psi}\]
    since $\d (f_{q})_{\psi^{-1}(q)} = id_{\R^m}$.
    We recall that
    $$(\pi_{q}^T)^{-1} |_{V_q} = \psi \circ f_{q}^{-1} \circ d \big(\psi^{-1}\big) _q \circ t_{-q} |_{V_q} $$
    and that $\d (\psi^{-1})$, $d( f_{q}^{-1})$ are $\alpha$-Hölder on receptively $\psi(U \cap M)$ and $f_{q}(W)$.
    Thus $\d \big( (\pi_{q}^T)^{-1} \big) |_{U \cap M}$ is $\alpha$-Hölder on $V_{q}$ and its Hölder constant is bounded by a constant that does not depend on $q$.
\end{proof}



\subsection{Proof of Theorem~\ref{theo: majoration_fonction_critique}}\label{subsec:profftheme2}

In order to prove the Theorem \ref{theo: majoration_fonction_critique}, we will need the following results. 

\begin{lemma} \label{lem: intersection_variete_boule}
    Let $M$ be a $C^1$ submanifold of dimension $m$ in $\E^n$ and $p \in M$.
    There exists $R > 0$ such that for all $0 < r\leq R$,
    $M \cap \overline{B(p,r)}$
    is a $C^1$ submanifold of $\E^n$ with boundary.
\end{lemma}

\begin{proof}
    Let $g$ be the restriction  to $M$ of the map $x\mapsto \|x-p\|^2$.  The intersection $M \cap \overline{B(x,r)}$ is the sublevel set $\{g(x)\leq r\}.$ It is well-known that if $r$ is a regular value of the $C^1$ fonction $g$, then $\{g(x)\leq r\}$ is a $C^1$ submanifold with boundary $g^{-1}(r)$ (see \cite{Milnor1965}, Lemma 3 p12). The number $r$ is a regular value if, for all $x\in g^{-1}(r)$, we have: ${\rm rank\,} \d g_{x}=1$. Let $F:U\subset T_pM\to M\subset \E^n$, $u\mapsto F(u)=(u,f(u))$ be a $C^1$ local cartesian parametrization of $M$ such that $F(0)=p$ and $\d f_{0}=0.$ On these coordinates, the map $g$ and its gradient read as follow 
    \[g(u)=\|u\|^2+\|f(u)\|^2\quad\mbox{and}\quad {\rm grad}\, g (u) =2u+2[(Jac\,f)(u)]^T(f(u))\]
    Since $f$ is $C^1$, we have $f(u)=o(\|u\|)$ and $\lim_{u\to 0}(Jac\,f)(u)=0$. 
    Therefore
    \[2[(Jac\,f)(u)]^T(f(u))=o(\|u\|).\]
    This shows that, when $r>0$ is small enough, ${\rm grad}\, g (u)\neq 0.$ The statement of the lemma follows. 
\end{proof}

\noindent We recall that $M^d = \{ x \in \E^n ~|~ \d_M(x)  \leq d \}$.

\begin{proposition} \label{prop: lien_C_1alpha_grad_g}
Let $M$ be a $C^{1,\alpha}$ compact submanifold without boundary of $\E^n$.
There exist $d > 0$ and $C > 0$ such that for any $p \in \mathscr{M}(M) \cap M^d$ and $q_1$ and $q_2$ two projections of $p$ on $M$ that realize the diameter of $\Gamma_K(p)$ then we have
\[ 1 - \| \nabla \d_M(p) \|^2 \leq C \, \|q_1-q_2\|^{2 \alpha} .\]
\end{proposition}

\begin{proof}
According to Proposition \ref{prop: graph_constant_holder_global}, by compactness of $M$, there exist $R > 0$ and $K > 0$ such that, if $q_1 , q_2 \in M$ verify 
$$\|q_1 - q_2 \| < R$$ 
then $M \cap B(q_1,R)$ is locally the graph of a $C^{1,\alpha}$ function $h : V \cap T_{q_1} M \to N_{q_1} M$ \textit{i.e.}
\[ M \cap B(q_1,R) = \big\{ (x,h(x)) ~|~x \in V \subset T_{q_1} M \big\} ~~ \text{and} ~~ \d h_{x_1} = 0\]
with $q_1 = (x_1,0)$ and $V$ an open set of $T_{q_1} M$.
Moreover $\d h$ is $\alpha$-Hölder on $V$ and the $\alpha$-Hölder constant of $\d h$ is smaller than $K$.
Furthermore, for $R$ small enough $ M \cap B(q_1,R)$ and $V$ are arcwise connected sets.

Let $p \in \mathscr{M}(M)$ and $q_1$ and $q_2$ two projections of $p$ on $M$, that realize the diameter of $\Gamma_K(p)$. We have $$\frac{1}{2} \|q_1 - q_2\| \leq \d_M(p).$$
Thus, for the rest of the proof, we are going to suppose that $d_M(p) \leq d :=  \frac{R}{2}$. \\

\begin{figure}[!ht]
\centering
\includegraphics[width=0.50\linewidth]{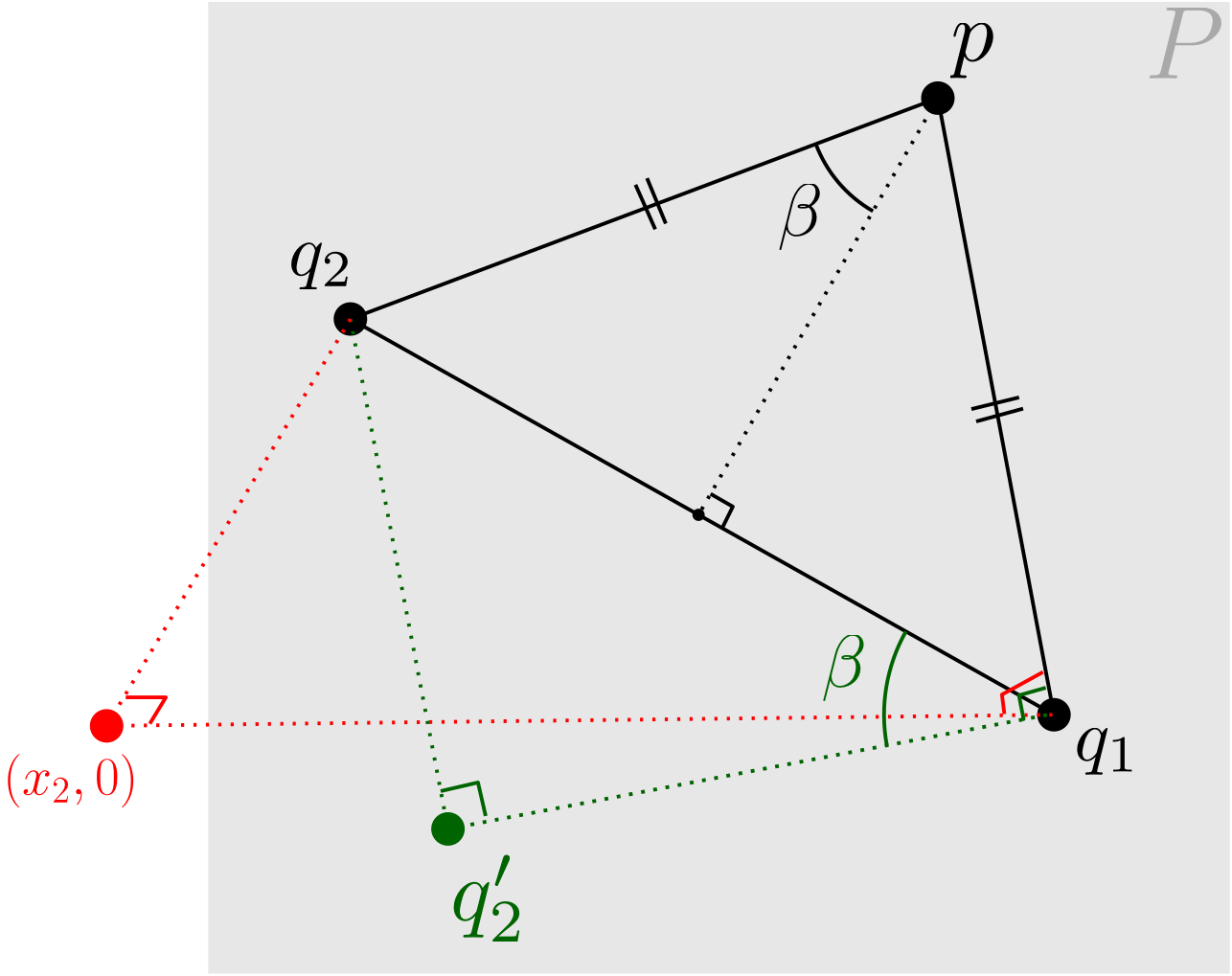}
\caption{\centering The construction of the proof.}
\label{fig: regularite_c_1_alpha}
\end{figure}

Let $h : V \subset T_{q_1}M \to N_{q_1}M$ as described above and we denote $q_1 = (x_1,0)$ and $q_2 = (x_2 , h(x_2))$ with $x_2 \in V$ (see Figure \ref{fig: regularite_c_1_alpha}).
To prove our result, we need to establish three inequalities.
The map $h$ is a $C^{1,\alpha}$ function and $\d h$ is $\alpha$-Hölder on $V \subset T_{q_1}M$ of constant $K$.
Using Lemma \ref{lem: intersection_variete_boule}, we have that for $R$ small enough $\overline{B(q_1,R)} \cap M$ is a compact connected $C^1$ submanifold.
Since $V$ is the image by $\pi_{q_1}^T$ of $B(q_1,R) \cap M$, we have that $\overline{V}$ is a connected compact $C^1$ submanifold of $\E^n$.
If we apply Lemma \ref{lem: carac_c_1_alpha_2}, to $h |_{V}$ , we have
\begin{align*}
   \| (x_2,0) - q_2 \| &= \| h(x_2) \| =\| h(x_2) - h(x_1) - \d h_{x_1}.(x_2-x_1) \| \leq  \frac{C_h \lambda_M}{1+\alpha} \| x_2 - x_1 \|^{1+\alpha}   
\end{align*}
since $h(x_1) = 0$ and $\d h_{x_1} = 0$ with $\lambda_M$ the constant of Lemma \ref{lem: distance_geodesique_norme_R_n} of $M$.
Thus, we obtain our first inequality
\begin{equation} \label{eq: inegalite_1}
     \| (x_2,0) - q_2 \| \leq  K' \| (x_2,0) - q_1\|^{1+\alpha}.
\end{equation}
with $K' = \frac{K\, \lambda_M}{1+\alpha}$.

\noindent Second, we define $\mathscr{P}$ the affine plane containing $p$, $q_1$ and $q_2$ and we denote by $q_2'$ the orthogonal projection of $q_2$ onto $\mathscr{P}$ orthogonal to ${pq_1}$ of $(x_2,0)$ (see Figure \ref{fig: regularite_c_1_alpha}). 
We remark that since $(x_2,0) - q_1 \in T_q M$ and ${pq_1} \in N_{q_1} M$ (see Lemma \ref{lem_normale_distnace_min_variete}) then the points $(x_2,0)$, $q_1$ and $p$ formed a right angle.
Thus since ${q_2 q_2'}$ is the orthogonal projection of $(x_2,0) - q_2$, we get our second inequality
\begin{equation} \label{eq: inegalite_2}
    \| q_2' - q_2 \| \leq \| (x_2,0) - q_2 \|.
\end{equation}
For the last inequality, we remark that $q_2'$, $q_1$ and $q_2$ forms a right triangle at $q_2'$ and that $(x_2,0)$, $q_1$ and $q_2$ formed also a right triangle at $(x_2,0)$, thus using the Pythagorean theorem, we get
\[ \|q_2 - q_2' \|^2 + \| q_2' - q_1 \|^2 =  \|(x_2,0) - q_2 \|^2 + \| (x_2,0) - q_1 \|^2\]
thus, using Inequality \eqref{eq: inegalite_2}, we get
\begin{equation} \label{eq: inegalite_3}
     \| (x_2,0) - q_1 \| \leq  \| q_2' - q_1 \|.
\end{equation}
Putting together Inequality \eqref{eq: inegalite_1}, \eqref{eq: inegalite_2} and \eqref{eq: inegalite_3}, we have
\[  \| q_2' - q_2 \| \leq \| (x_2,0) - q_2 \| \leq   K' \| (x_2,0) - q_1 \|^{1+\alpha} \leq  K' \| q_2' - q_1 \|^{1+\alpha} \]
but since $\| q_2' - q_1 \| \leq \| q_2 -q_1 \|$, we get
\[  \| q_2' - q_2 \| \leq  K \| q_2 - q_1 \|^{1+\alpha}.\]
Now, we recall that $p$, $q_1$ and $q_2$ form an isosceles triangle and we denote by $\beta \in [0,\frac{\pi}{2}[$ the semi-angle formed by $q_1$, $p$ and $q_2$.
We have
\[ \sin(\beta) = \frac{\| q_2'-q_2\|}{\|q_1-q_2\|} \leq  K' \| q_2 - q_1 \|^{\alpha}. \]
Finally, by Lemma \ref{lem: consequnce_Jung}, we get
\[1 - \| \nabla \d_M(p) \|^2 \leq \frac{2}{1 + \frac{1}{n}} \sin^2(\beta) \leq \underbrace{\frac{2 (K')^2}{1 + \frac{1}{n}}}_{C} \|q_1-q_2\|^{2\alpha}.\]
\end{proof}


\begin{proof}[Proof of Theorem \ref{theo: majoration_fonction_critique}]
Let $p \in \E^n \setminus M$. We can assume that $p$ belongs to the medial axis of $M$, otherwise the result is straightforward.
Let $q_1$ and $q_2$ be two projections that realize the diameter of the set of projections of $p$ on $M$.
According to Proposition \ref{prop: lien_C_1alpha_grad_g}, there exist $d_0 > 0$ and $C' > 0$ (both dependent of $M$ only) such that if $\d_M(p) < d_0$ then 
\[ 1 - \| \nabla \d_M(p) \|^2 \leq C' \, \|q_1-q_2\|^{2 \alpha} .\]
From now on, we suppose that $\d_M(p) < d_0$.
Using the triangular inequality
$$\| q_1 - q_2\| \leq \| q_1 - \Omega_M(p)\| + \| q_2 -  \Omega_M(p)\| \leq 2 ~ \mathscr{R}_M(p)$$
then by Proposition \ref{prop: expression_grad_g_rayon}
$$\| q_1 - q_2\|^2 \leq 4 ~\mathscr{R}_M(p)^2 = 4 d_M(p)^2 \Big( 1 - \vert \vert \nabla \d_M(p) \vert \vert^2 \Big).$$ 
Thus, we get that
$$ 1 - \| \nabla d_M(p) \|^2  \leq 4 ^\alpha C' ~d_M(p)^{2 \alpha} \Big( 1 - \| \nabla \d_M(p) \|^2 \Big)^{\alpha}.$$
Assume now that $\d_M(p) \geq d_0$, since $ \| \nabla \d_M(p) \| \leq 1$ it follows that
$$1 - \| \nabla \d_M(p) \|^2 \leq \big( 1 - \| \nabla \d_M(p) \|^2 \big)^\alpha.$$ 
Then by setting $C = \max \big( 4 ^\alpha C' , \frac{1}{d_0^{2 \alpha}} \big)$, we have, for all $p \in \E^n \setminus M$,
$$ 1 - \| \nabla d_M(p) \|^2  \leq C ~d_M(p)^{2 \alpha} \Big( 1 - \| \nabla \d_M(p) \|^2 \Big)^{\alpha}.$$
In particular, if $\alpha < 1$, one has
$$ 1 - \| \nabla \d_M(p) \|^2  \leq C^{\frac{1}{1 - \alpha}} ~ \d_M(p)^{\frac{2 \alpha}{1- \alpha}}$$
and if $\alpha = 1,$
$$ 1 - \| \nabla d_M(p) \|^2  \leq C ~d_M(p)^{2} \Big( 1 - \| \nabla \d_M(p) \|^2 \Big).$$
Whenever $\d_M(p) < \frac{1}{\sqrt{C}}$, the above inequality can hold only if
$$1 - \| \nabla d_M(p) \|^2 = 0.$$
\end{proof}

\begin{proof}[Proof of Corollary~\ref{coro: majoration_fonction_critique}]
By taking the lower bound in Theorem \ref{theo: majoration_fonction_critique}, we get for all $d > 0$
$$ 1 - \chi_M(d)^2 \leq C ~ d^{\frac{2 \alpha}{1- \alpha}}.$$
We know that $\chi_M(0) = 1$ and that $0 \leq 1 - \chi_M(d)$ thus we have
$$0 \leq \big( 1+\chi_M(d) \big) \frac{1- \chi_M(d)}{d} \leq C ~ d^{\frac{3\alpha - 1}{1- \alpha}}.$$
Hence, when $d$ tends towards $0$, we obtain the desired result.
\end{proof}

\section{Sharpness of the inequality in Theorem~\ref{theo: majoration_fonction_critique}} \label{sec: Deux_exemples}
The goal of this section is to show the sharpness of the exponent that appears in Theorem~\ref{theo: majoration_fonction_critique}. For this purpose, we study the family of curves 
\[
\C_\alpha:=\{(t,|t|^{1+\alpha}),t\in \R\},
\]
where $\alpha\in(0,1)$. Note that each curve $\C_\alpha$ is of class $C^{1,\alpha}$, but not of class $C^{1,\alpha'}$ for any $\alpha'>\alpha$. The main result is the following proposition that states that the exponent of Theorem~\ref{theo: majoration_fonction_critique} is reached for these curves.
\begin{proposition} \label{prop: equivalent_fct_critique_cas_pair}
Let $\alpha\in [0,1[$. The critical function $\chi :\R\to \R$ of the curve $\C_\alpha$ satisfies
 $$1 - \chi^2(d) \underset{d \to 0}{\sim} (1+\alpha )^{ \frac{2}{1-\alpha}} d^\frac{2 \alpha}{1 - \alpha}.$$
\end{proposition}
Note that Theorem~\ref{theo: majoration_fonction_critique} holds for compact submanifolds with no boundary. The curves $\C_\alpha$ are not compact. However we can restrict each $\C_\alpha$ to any interval $[-a,a]$ and connect its two endpoints with a regular curve of class $C^2$, in order to build a compact $1$-dimensional manifold without boundary, whose critical function is the same that the one of $\C_\alpha$.\\

\noindent \textbf{Geometric interpretation.} Given a point $p \in  \mathscr{M}(\C_\alpha)$, we prove below that it has only two projections $q_1,q_2$ that are symmetrical with respect to the vertical axis. 
We denote by $\beta = \beta_p \in [0,\frac{\pi}{2}[$ the semi-angle formed by $q_1$, $p$ and $q_2$.
It is readily checked that this angle is linked to the generalised gradient by the formula
$$ \| \nabla \d_{\C_\alpha}(p) \| = \cos \beta.$$
Since all the curves $\C_\alpha$ are of class $C^1$, by Theorem \ref{theo: C1_implique_continuite_fonction_crit}, when $p$ tends to $\C_\alpha$, the angle $\beta$ tends to $0$. 
Proposition~\ref{prop: equivalent_fct_critique_cas_pair} states that for those curves $\C_\alpha$, the larger $\alpha$ is, the faster this angle $\beta$ tends to zero. 
This is illustrated in Figure~\ref{fig: graph_fonction_exemple_pair}.
\begin{figure}[H]
\begin{subfigure}[b]{0.49\textwidth}
\centerline{\includegraphics[width=1\linewidth]{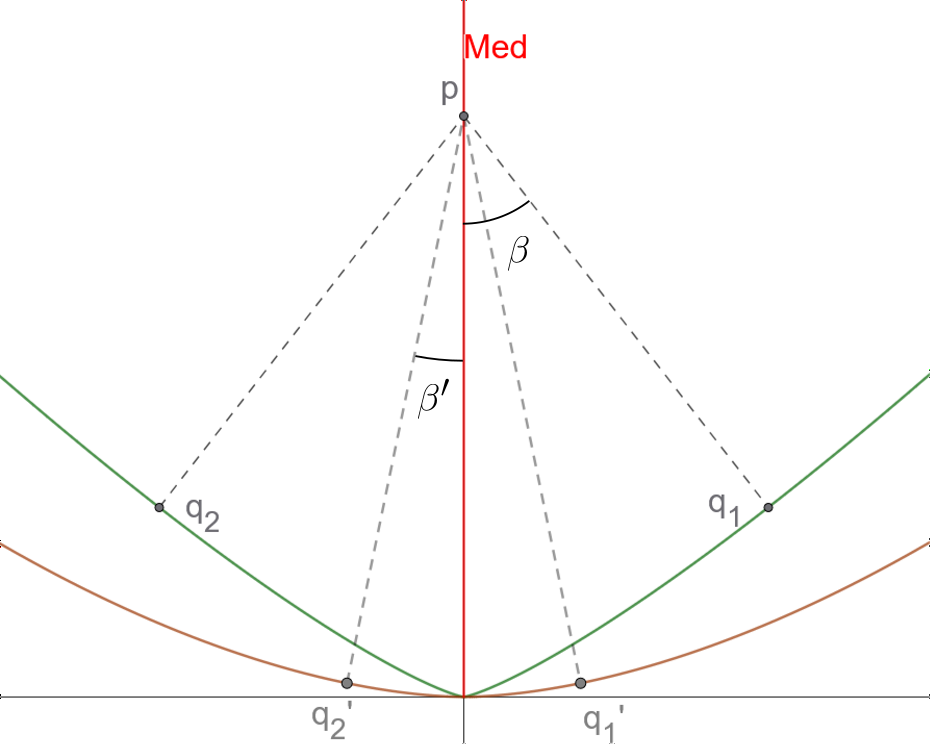}}
\caption{\centering Curve $\C_\alpha$ for $\alpha = 5/4$ in green and $\alpha = 7/4$ in red. $q_1$, $q_2$ (\textit{resp.} $q_1'$, $q_2'$) are the projections of $p$ on the green (\textit{resp.} red) curve. The angle $\beta$ depends on the regularity of the curve.}
\end{subfigure}
\hfill
\begin{subfigure}[b]{0.49\textwidth}
\centerline{\includegraphics[width=1\linewidth]{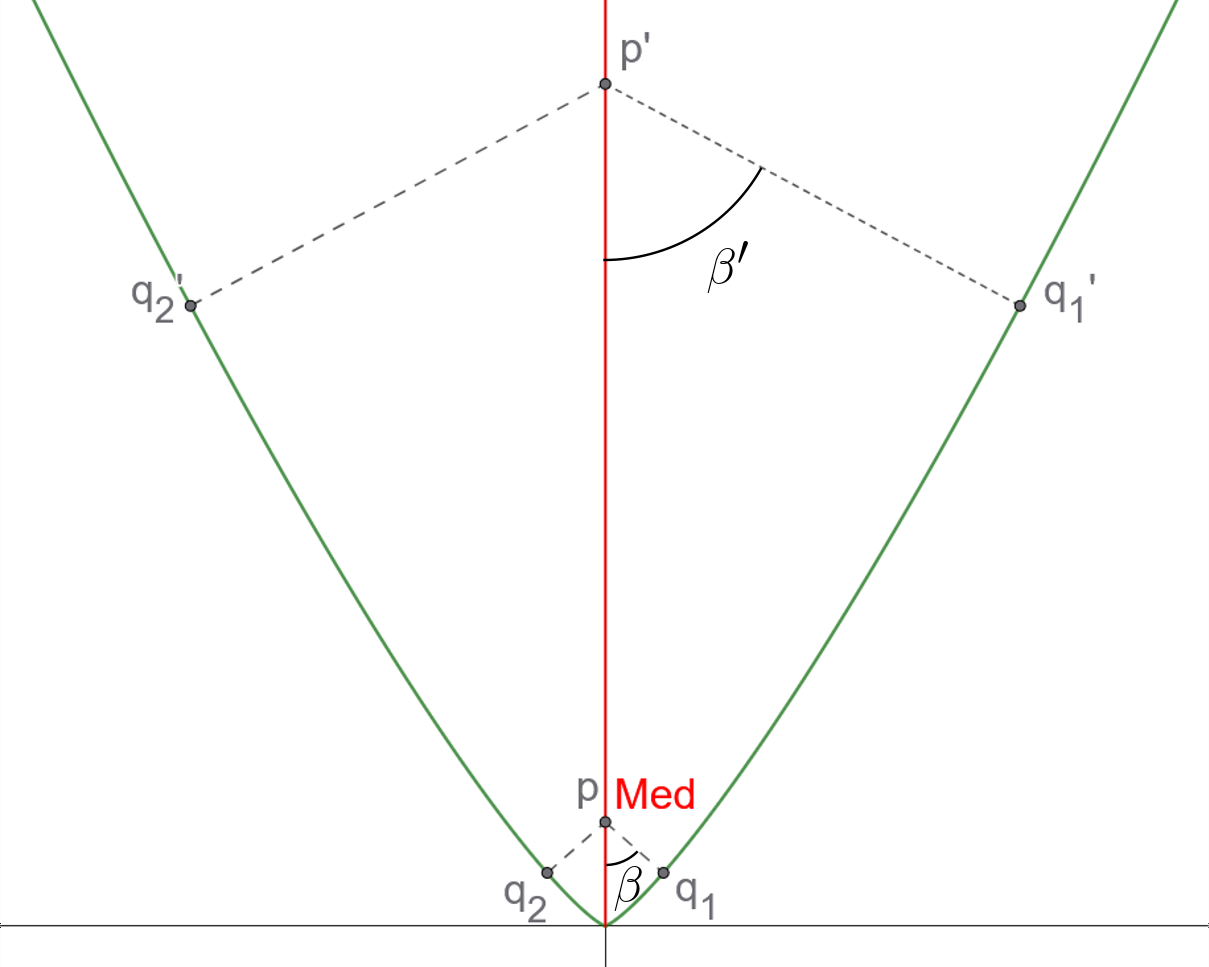}}
\caption{\centering Curve $\C_\alpha$ with $\alpha = 5/4$ (green) and its medial axis (red). The angle $\beta$ tend to zero when the distance converge to $0$. (Be aware of the zoom out with respect to (a))}
\end{subfigure}
\caption[]
{\centering Different curves $\C_\alpha$ and their medial axis.}
\label{fig: graph_fonction_exemple_pair}
\end{figure}

\begin{corollary} \label{coro: equivalent_derive_fonction_crit_cas_pair} 
Let $\alpha\in [0,1[$. The derivative of the critical function $\chi:\R\to \R$ of the curve $\C_\alpha$ around $0$ satisfies
$$\chi'(d) \underset{d \to 0}{\sim} - \frac{\alpha (1 + \alpha)^{\frac{2}{1-\alpha}}}{1-\alpha} \, d^\frac{3 \alpha - 1}{1-\alpha}.$$ 
\end{corollary}

Note that the derivative of the critical function at zero has a very particular behavior reminiscent of the Hausdorff dimension for fractal objects. Indeed, we see from Corollary~\ref{coro: equivalent_derive_fonction_crit_cas_pair} that the derivative of the critical function $\chi:\R\to \R$ of the curve $\C_\alpha$ satisfies
\begin{align*}
    \chi'(0) = \left\{ \begin{array}{ll}
        -\infty &\text{if} ~~ \alpha < \sfrac{1}{3} \\
        - C &\text{if} ~~ \alpha = \sfrac{1}{3} \\
        0 &\text{else}
    \end{array} 
   \right.
\end{align*}
with $C = - \frac{1}{2}\big(\frac{4}{3}\big)^3$. One can also observe the same behavior in the numerical experiments shown in Figure~\ref{fig: graphe_fonction_critique_question_comparaison_2}.

\begin{figure}[!ht]
\centering
\includegraphics[width=0.9\linewidth]{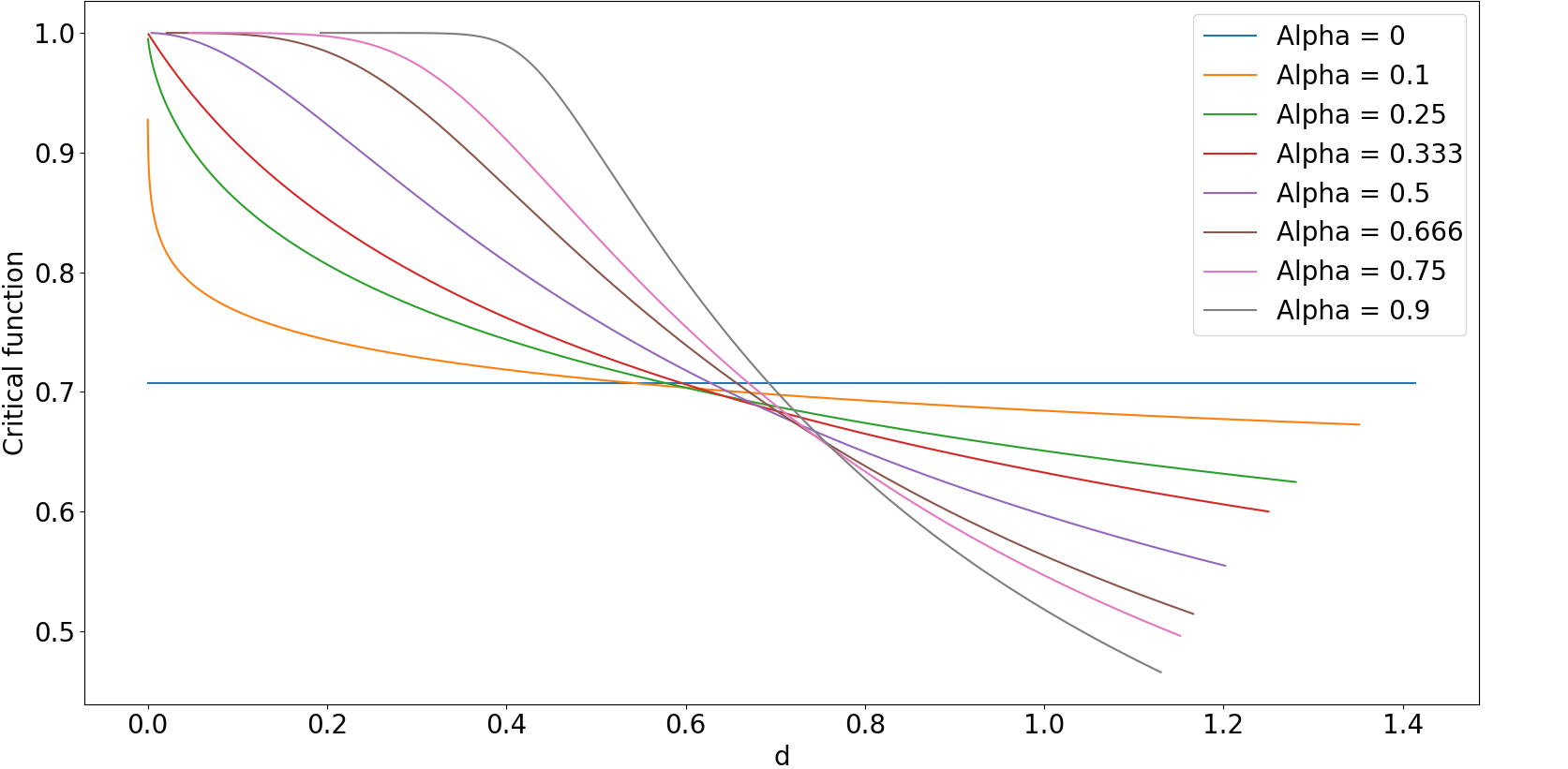}
\caption{\centering Computer simulation of the graphs of the critical functions $\chi$ of the curve $\C_\alpha$,  for different values of $\alpha$.}
\label{fig: graphe_fonction_critique_question_comparaison_2}
\end{figure}

The end of this section is devoted to the proof of Proposition~\ref{prop: equivalent_fct_critique_cas_pair} and of Corollary~\ref{coro: equivalent_derive_fonction_crit_cas_pair}.

\subsection{Preliminary results on the medial axis}

\begin{lemma} \label{lem: medial_axis_moitie_de_courbe}
Let $f : t \in I \to t^{1+\alpha}$ with $\alpha \in [0,1[$ and $I$ an open interval in $\R_+^*$.
Let $K = \mathrm{graph}(f)$ then $\mathscr{M}(K) = \emptyset$.
\end{lemma} 

\begin{proof}
We suppose that the result does not hold, \textit{i.e.} that $\mathscr{M}(K) \neq \emptyset$. Therefore there exists $p = (u,v) \in \mathscr{M}(K)$ and $q_1 = (t_1, f(t_1))$, $q_2 = (t_2, f(t_2))$ are two projections of $p$ onto $K$, with $t_1\neq t_2$. We can  assume that $t_2 > t_1 > 0$.
We have that $\d_K(q_1) = \d_K(q_2)$ thus it verifies
\begin{equation} \label{eq: 1er_equation_systeme_lin}
    (t_1-u)^2 + (f(t_1)-v)^2 = (t_2-u)^2 + (f(t_2)-v)^2.
\end{equation}
Expanding the equation, we get a linear equation in $u$ and $v$
$$2 \big( t_1 - t_2 \big) \, u + 2 \big( f(t_1) -f(t_2) \big) \, v + \Big[ t_2^2 - t_1^2 + (f^2(t_2) - f^2(t_1)) \Big] = 0 .$$
According to Lemma \ref{lem_normale_distnace_min_variete}, since $p-q_i$ is orthogonal to $(1,f'(t_i))$, we have
\begin{equation} \label{eq: 2eme_equation_systeme_lin}
    (t_1 - u) + f'(t_1) (f(t_1)- v) = 0 ~~ \text{and} ~~ (t_2 - u) + f'(t_2) (f(t_2)- v) = 0.
\end{equation}
Thus $(1\ u\ v)^T$ is a solution of the following linear system of 3 equations:
\[ \begin{pmatrix}
    (t_2^2-t_1^2)+(f^2(t_2)-f^2(t_1)) & 2(t_1-t_2) & 2(f(t_1)-f(t_2))\\
 -t_1-f'(t_1)f(t_1) & 1 &  f'(t_1)\\
 -t_2-f'(t_2)f(t_2) & 1 &  f'(t_2)
\end{pmatrix}
\begin{pmatrix}
    1 \\ u \\ v
\end{pmatrix}
= 
\begin{pmatrix}
    0 \\ 0 \\ 0
\end{pmatrix}
\]
In order for this system to have a solution, the determinant of the matrix must vanish. Expanding the expression of that determinant with respect to the first row, we obtain
\begin{multline*}
    - t_1^{2+3\alpha} . (1 + \alpha) + t_1^{2+2 \alpha} t_2^\alpha . (1 + \alpha) (1+2 \alpha) - t_1^{2+\alpha} . (1 - \alpha) + t_1^2 t_2^\alpha . (1 + \alpha) - t_1^{1 + 2 \alpha} t_2^{1+\alpha} .(1 + \alpha) 2 \alpha\\
    - t_1^{1+\alpha} . \big( t_2^{1+2\alpha} . (1 + \alpha) 2 \alpha  + t_2 . 2 \alpha\big) 
    - t_1 t_2^{1+\alpha} . 2 \alpha + t_1^\alpha . \big( (1 + \alpha)  t_2^2 + (1 + \alpha) (1+2 \alpha) t_2^{2+2\alpha} \big) \\
     - \big( (1 - \alpha)  t_2^{2+\alpha} + (1 + \alpha) t_2^{2+3\alpha} \big) = 0.
\end{multline*}
By grouping terms of the same power and factor, we get
\begin{multline*}
    (1+ \alpha) \Big[  t_1^{2+2\alpha}t_2^{\alpha} + t_1^{\alpha} t_2^{2+2\alpha} - t_1^{2+3\alpha} - t_2^{2+3\alpha} \Big] + (1+ \alpha) 2 \alpha \Big[  t_1^{2+2\alpha}t_2^{\alpha} + t_1^{\alpha} t_2^{2+2\alpha} - t_1^{1 + 2 \alpha} t_2^{1+\alpha} -  t_1^{1+\alpha} t_2^{1+2\alpha} \Big] \\
    + (1 + \alpha) \Big[  t_1^{2}t_2^{\alpha} + t_1^{\alpha} t_2^{2} - t_1^{2+\alpha} - t_2^{2+\alpha} \Big]
    + 2 \alpha \Big[  t_1^{1+\alpha} t_2 + t_1 t_2^{1+\alpha} -t_1^{2+\alpha} - t_2^{2+\alpha} \Big] = 0
\end{multline*}
Factoring, this gives us
\begin{align*}
    &(1+\alpha) \left[\big(  t_2^{2+2\alpha} - t_1^{2+2\alpha} \big) \big( t_1^{\alpha} - t_2^{\alpha} \big) 
    -  2 \alpha \big( t_1^\alpha t_2^{1+2\alpha} - t_1^{1+2\alpha} t_2^{\alpha}  \big) \big( t_1 - t_2 \big)\right] \\
    & + (1+\alpha) \big( t_2^{2} - t_1^{2} \big) \big( t_1^{\alpha} - t_2^{\alpha} \big)
    - 2 \alpha \big( t_2^{1+\alpha} - t_1^{1+\alpha} \big) \big( t_1 - t_2 \big)= 0
\end{align*}
To complete the proof, it suffices to check that, for all $t_2 > t_1$, the following two inequalities:
\begin{align*}
    \big(  t_2^{2+2\alpha} - t_1^{2+2\alpha} \big) \big( t_2^{\alpha} - t_1^{\alpha} \big) &> 2 \alpha \big( t_2^{1+2\alpha} t_1^{\alpha} - t_2^\alpha t_1^{1+2\alpha} \big) \big( t_2 - t_1 \big) \\
    (1+\alpha) \big( t_2^{2} - t_1^{2} \big) \big( t_2^{\alpha} - t_1^{\alpha} \big) &> 2 \alpha \big( t_2^{1+\alpha} - t_1^{1+\alpha} \big) \big( t_2 - t_1 \big)
\end{align*}
which are equivalent to
\begin{align*}
	t_2^{1+2 \alpha} - t_1^{1+2\alpha} > \big( 1+2 \alpha \big) \Big( t_2^{1+\alpha} t_1^\alpha - t_2^\alpha t_1^{1+\alpha} \Big) \\
    (1+\alpha) \big( t_2 + t_1 \big) \big( t_2^{\alpha} - t_1^{\alpha} \big) > 2 \alpha \, \big( t_2^{1+\alpha} - t_1^{1+\alpha} \big).
\end{align*}
These inequalities are proved in Lemma \ref{lem: cas_pair_inegalite_intermediare}.
\end{proof}

\begin{lemma} \label{lem: cas_pair_inegalite_intermediare}
Let $\alpha \in [0,1[$ and $t_2 > t_1 \geq 0$ then 
\begin{align*}
	 t_2^{1+2 \alpha} - t_1^{1+2\alpha} > \big( 1+2 \alpha \big) \Big( t_2^{1+\alpha} t_1^\alpha - t_2^\alpha t_1^{1+\alpha} \Big) \\
     (1+\alpha) \big( t_2 + t_1 \big) \big( t_2^{\alpha} - t_1^{\alpha} \big) > 2 \alpha \, \big( t_2^{1+\alpha} - t_1^{1+\alpha} \big)
\end{align*}
\end{lemma}

\begin{proof}
For the first inequality, let $x \in [t_1, \infty [$ and
\[ f(x) := x^{1+2\alpha} - t_1^{1+2\alpha} - ( 1+2 \alpha ) ( x^{1+\alpha} t_1^\alpha - x^\alpha t_1^{1+\alpha} ).\]
Differentiating, we get
\[f'(x) = (1+2 \alpha) \, x^{2\alpha} - (1+\alpha) ( 1+2 \alpha ) \, x^{\alpha} t_1^\alpha + \alpha (1+2 \alpha) \, \frac{t_1^{1+\alpha}}{x^{1-\alpha}} . \]
We define
\[ g(x) := x^{1-\alpha} f'(x) = (1+2 \alpha)\, x^{1+ \alpha} - (1+\alpha) ( 1+2 \alpha ) \, x t_1^\alpha + \alpha (1+2 \alpha) \, t_1^{1+\alpha}  \]
and by differentiating
\[g'(x) = (1+\alpha) (1+2\alpha) \big( x^\alpha - t_1^\alpha \big).\]
For all $x > t_1$, we have $g'(x) > 0$.
Thus $g'$ is a strictly increasing function and since $g(t_1) = 0$, we have that $g$ is strictly positive on $ ]t_1 , + \infty[$.
The same goes for $f'$ and thus $f$ is strictly increasing on $ ]t_1 , + \infty[$. 
Since $f(t_1) = 0$, $f(x) > 0$ for all $x > t_1$. 
Therefore the first inequality is true.
The other inequality is shown similarly.
\end{proof}

\noindent Using Lemma \ref{lem: medial_axis_moitie_de_courbe}, we can fully determine the medial axis.

\begin{lemma} \label{lem: medial_axis_exemple_fonction_pair}
For every $\alpha\in [0,1[$, the medial axis of the curve $\C_\alpha$ is given by
$$\mathscr{M}(\C_\alpha) = \{ 0 \} \times \R_+^*.$$   
Furthermore, every point $p\in\mathscr{M}(\C_\alpha)$ has exactly two projections $q_1\neq q_2$ that are symmetric with respect to the axis $x=0$.
\end{lemma}

\begin{proof}
Let $p = (u,v) \in \mathscr{M}(\C_\alpha)$ and $q_1 = (t_1 ,f(t_1))$, $q_2 = (t_2, f(t_2))$ be two projections of $p$ on $K$. By symmetry of the curve and by  Lemma~\ref{lem: medial_axis_moitie_de_courbe}, we can assume that $t_1 > 0$ and $t_2 \leq 0$.

We are first going to show that $t_2<0$. We suppose that $t_2 = 0$, which implies that the normal to $q_2 = (0,0)$ is vertical. Therefore $p$ must be in $\{0\} \times \R_+^*$, thus $u=0$. By Equations~\eqref{eq: 2eme_equation_systeme_lin} and~\eqref{eq: 1er_equation_systeme_lin}, we have 
\begin{align*}
\left\{ \begin{array}{l}
    \displaystyle t_1^2 + (f(t_1)-v)^2 =  v^2  \\
    \displaystyle t_1  + f'(t_1) (f(t_1)- v) = 0\\
\end{array} 
\right.
\end{align*}
Since  $f(t)=t^{1+\alpha}$ on $\R^*_+$ and $t_1>0$, we get
\begin{align*}
\left\{ \begin{array}{l}
    \displaystyle t_1^2 + t_1^{2(1+\alpha)}-2v \, t_1^{1+\alpha} =  0 \\
    \displaystyle t_1  + (1+ \alpha)\, t_1^{1+2\alpha} - (1+ \alpha) v \, t_1^{\alpha} = 0\\
\end{array} 
\right. ,
\end{align*}
that implies after computation
\[ 0 < t_1^{2 \alpha} = 1 - \frac{2}{1 + \alpha} \leq 0  \]
which is absurd.
Therefore, we have $t_2 < 0$. 

\noindent Now, we want to prove that $p=(u,v) \in  \{0\} \times \R_+^*$. Let us first suppose that $u \neq 0$. By symmetry of $f$, we can assume $u>0$.
Thus, we have
\begin{multline*}
d_{\C_\alpha}(p) = \| p - q_2 \| = \sqrt{(t_2-u)^2 + (f(t_2)-v)^2} \\
> \sqrt{(-t_2-u)^2 + (f(t_2)-v)^2} = \| (-t_2, f(-t_2)) -  p \| \geq d_{\C_\alpha}(p),
\end{multline*}
which is a contradiction. 
According to Equation \eqref{eq: 2eme_equation_systeme_lin}, we have
\[ t_1 + f'(t_1) \big( f(t_1) - v) = 0\]
thus $v>f(t_1)>0$ since $t_1>0$.
We have shown that $(u,v) \in  \{0\} \times \R_+^*$. \\

\noindent For the converse, let $p = (0,v) \in \{ 0 \} \times \R_+^*$. The set $\Gamma_{\C_\alpha}(p)$ of projections of $p$ is given by the minimizers of  the following function
$$ g : t \in \R \longmapsto \|p - (t,f(t))\|^2  = t^2 + \big( t^{1+\alpha} - v \big)^2.$$
We find out that the minimum of $g$ is reached for  $t_{min} > 0$ and $-t_{min}$. More precisely, we have $t_{min} = h^{-1}(v)$ where $$
h : t \in \R_+ \longmapsto \frac{t^{1-\alpha}}{1+\alpha} + t^{1+\alpha} \in \R_+
$$
is a one-to-one map.
This implies that $\Gamma_{\C_\alpha}(p)$ contains the two points $q_1 = (t_{\min},f(t_{\min}))$ and $q_2 = (-t_{\min},f(-t_{\min}))$. Hence $p\in \mathscr{M}(\C_\alpha)$.
\end{proof}



\subsection{Proof of Proposition \ref{prop: equivalent_fct_critique_cas_pair}}

\begin{lemma} \label{lem: fonction_pair_C1alpha}
Let $\alpha\in [0,1[$. The graph of the critical function $\chi:\R_+ \to \R$ of $\C_\alpha$ is parametrized by
\[
t\in \R^+ \mapsto  
\left( \begin{array}{l}
\tilde{d}(t)\\
\chi(\tilde{d}(t))\\
\end{array} 
\right)
   = 
   \left( \begin{array}{l}
        t \bigg( 1 + \frac{1}{f'(t)^2} \bigg)^{\frac{1}{2}}\\
       \bigg( f'(t)^2 + 1 \bigg)^{-\frac{1}{2}}\\
    \end{array} 
   \right),
\]
with $f : t \in \R \to |t|^{1+\alpha}$.
\end{lemma}

\begin{proof}
By Lemma \ref{lem: medial_axis_exemple_fonction_pair}, for every $p= (0,v) \in \mathscr{M}(\C_\alpha)$, we have $\Gamma_{\C_\alpha}(p) = \{q_1,q_2\}$, with $t>0$, $q_1=(t,f(t)$ and $q_2=(-t,f(-t))$. We then get 
$$\Omega_{\C_\alpha}(p) = \frac{q_1+q_2}{2} = \big( 0 , f(t) \big).$$
Setting  $d = \d_{\C_\alpha}(p)$, we get
\begin{align*}
    \| \nabla \d_{\C_\alpha} (p) \|^2 = \frac{\|p - \Omega_{\C_\alpha}(p)\|^2}{d^2} = \frac{ \big( v - f(t) \big)^2}{d^2} .
\end{align*}
Equations~\eqref{eq: 2eme_equation_systeme_lin} and~\eqref{eq: 1er_equation_systeme_lin} give
$$t + f'(t) \big( f(t) - v \big) = 0
\quad \mbox{and}\quad d^2 = t^2 + \big( v - f(t) \big)^2$$
which implies
$$d^2 = t^2 \bigg( 1 + \frac{1}{f'(t)^2} \bigg)$$
and thus 
\[ \| \nabla \d_{\C_\alpha} (p) \|^2 = \frac{ \Big( \frac{t}{f'(t)} \Big)^2 }{t^2 + \Big( \frac{t}{f'(t)} \Big)^2} = \frac{1}{1 + f'(t)^2}. \]
We remark that the function $\tilde{d}:t \mapsto t^2 \bigg( 1 + \frac{1}{f'(t)^2} \bigg)$ is strictly monotone and therefore bijective from $\R_+$ to $\R_+$.
Thus, for any $d > 0$, there is a unique associate value of $t$ and thus a unique point on the medial axis.
It follows that $$\| \nabla \d_{\C_\alpha}(p) \| = \chi(d).$$
\end{proof}

\begin{proof}[Proof of Proposition \ref{prop: equivalent_fct_critique_cas_pair}]
We see $d$ as a function of $t$ and denote it by $d=\tilde{d}(t)$.
By Lemma \ref{lem: fonction_pair_C1alpha} and from the expression of $f$, we deduce
$$\tilde{d}(t)^2  \underset{t \to 0}{\sim} \frac{1}{(1+\alpha)^2} \, t^{2(1-\alpha)}~~ \text{and} ~~  1 - \chi^2\big(\tilde{d}(t)\big) \underset{t \to 0}{\sim} (1+\alpha)^2 t^{2 \alpha}.$$
Elevating $\tilde{d}(t)$ to the power $\frac{2 \alpha}{1 - \alpha}$, we obtain according to the operating rules for equivalents
$$t^{2 \alpha}  \underset{t \to 0}{\sim} ( 1 + \alpha )^\frac{2 \alpha}{1-\alpha} \, \tilde{d}(t)^{ \frac{2 \alpha}{1-\alpha}}.$$ 
We then replace in the expression of $\chi^2$:
$$1 - \chi^2\big(\tilde{d}(t)\big) \underset{t \to 0}{\sim} (1+\alpha)^{ \frac{2}{1-\alpha}} \tilde{d}(t)^{ \frac{2 \alpha}{1 - \alpha}}.$$
Since the limit of $\tilde{d}$ is zero when $t$ tends to zero, we also have
$$1 - \chi^2(d) \underset{d \to 0}{\sim} (1+\alpha)^{ \frac{2}{1-\alpha}} d^{ \frac{2 \alpha}{1 - \alpha}}.$$
\end{proof}

\begin{proof}[Proof of corollary \ref{coro: equivalent_derive_fonction_crit_cas_pair}]
According to Lemma \ref{lem: fonction_pair_C1alpha}, we have
\begin{align*}
    \left\{ \begin{array}{l}
       \chi(d) = \bigg( (1+\alpha)^2 t^{2 \alpha} + 1 \bigg)^{-\frac{1}{2}}\\
        d^2 = t^2 \bigg( 1 + \frac{1}{(1+\alpha)^2 t^{2 \alpha}} \bigg).
    \end{array} 
   \right.
\end{align*}
We then fix $\alpha$ and consider $t$ as a function of $d$.
We also suppose that $0 < \alpha < 1$.
By differentiating the first equation, we get
$$\chi'(d) = \frac{ - \alpha (1+ \alpha)^2}{( t^{2 \alpha} (1+\alpha)^2 + 1)^{\frac{3}{2}}} . t'(d) . \big(t(d)\big)^{2\alpha - 1}.$$
By doing the same with the square root of the second line of the system, we get
$$1 = t'(d) \Bigg( \frac{1}{\big(t(d)\big)^{\alpha}} \sqrt{ \big(t(d)\big)^{2 \alpha} + \frac{1}{(1+\alpha)^2 }} - \frac{\alpha}{ (1+\alpha)^2} \frac{1}{\big(t(d)\big)^{\alpha} \sqrt{\big(t(d)\big)^{2 \alpha} + \frac{1}{ (1+\alpha)^2}}}\Bigg)$$
Since $t(0) = 0$, we have 
$$t'(d) \underset{d \to 0}{\sim}  \frac{1}{\Big( \frac{1}{t(d)} \Big)^\alpha \bigg( \sqrt{\frac{1}{(1+\alpha)^2}} - \frac{\alpha}{(1+\alpha)^2} \frac{1}{\sqrt{\frac{1}{(1+\alpha)^2}}} \bigg) } \underset{d \to 0}{\sim} 
\frac{1+\alpha}{1 - \alpha}  \big(t(d)\big)^{\alpha}.$$
Now, we would like to have a equivalent of $t$ as a function of $d$. 
By using the second equation, we obtain 
$$\frac{d^2}{\big(t(d)\big)^2} (1+\alpha)^2 \big(t(d)\big)^{2 \alpha} =  (1+\alpha)^2\big(t(d)\big)^{2 \alpha} + 1$$
and thus we obtain the following equivalent
$$d \underset{d \to 0}{\sim} \frac{1}{(1+ \alpha)} \big(t(d)\big)^{1 - \alpha} ~~ \text{et} ~~ t(d) \underset{d \to 0}{\sim} \big((1+\alpha) d \big)^{\frac{1}{1-\alpha}} $$
Finally, we can find an equivalent of $\chi'$ at zero
$$\chi'(d) \underset{d \to 0}{\sim} - \frac{ \alpha (1+ \alpha)^3 }{1-\alpha} . \big(t(d)\big)^{3 \alpha - 1} \underset{d \to 0}{\sim} - \frac{ \alpha (1 + \alpha)^{3 + \frac{3 \alpha + 1}{1-\alpha}}}{1-\alpha} . d^\frac{3 \alpha - 1}{1-\alpha} \underset{d \to 0}{\sim} - \frac{\alpha (1 + \alpha)^{\frac{2}{1-\alpha}}}{1-\alpha} . d^\frac{3 \alpha - 1}{1-\alpha}.$$

\end{proof}


\bibliographystyle{plain}
\bibliography{Biblio}


\end{document}